\documentclass[a4paper,11pt]{article}
\usepackage{amsfonts}
\usepackage{amsthm}
\usepackage{amsmath}
\usepackage{latexsym}
\usepackage{amssymb}
\usepackage[ansinew]{inputenc}
\usepackage{graphicx}
\usepackage{dsfont}
\newcommand{\nui}[1]{\hskip 1pt \left\| \hskip -7.5pt \left| \hskip 4pt #1 
\hskip 4pt \right| \hskip -7.5pt \right\| \hskip 1pt}

\newtheorem{fed}{Definition}[section]
\newtheorem*{fed*}{Definition}
\newtheorem*{feds*}{Definitions}
\newtheorem{teo}[fed]{Theorem}
\newtheorem*{teo*}{Theorem}
\newtheorem{lem}[fed]{Lemma}
\newtheorem{cor}[fed]{Corollary}
\newtheorem{pro}[fed]{Proposition}
\theoremstyle{definition}
\newtheorem{rem}[fed]{Remark}

\newtheorem*{rems*}{Remarks}

\newtheorem{notas}[fed]{Notations}

\def\coma{\, , \, }

\def\py{\peso{and}}
\newcommand{\peso}[1]{ \quad \text{ #1 } \quad }

\def\n0{n_{ \text{\rm \tiny o}}}
\def\varp{\varphi}

\def\vfi{\varphi}

\newcommand{\IN}[1]{\mathbb {I} _{#1}}

\def\suml{\sum\limits}

\def\QEDP{\tag*{\QED}}

\def\bce{\begin{center}}
\def\ece{\end{center}}

\DeclareMathOperator{\FP}{FP\,}

\def\cO{{\mathcal O}}

\def\py{\peso{and}}
\def\rk{\text{\rm rk}}
\def\noi{\noindent}
\def\cF{\mathcal F}
\def\cG{\mathcal G}

\def\LLM{\convfs}

\def\QED{\hfill $\square$}
\def\EOE{\hfill $\triangle$}
\def\EOEP{\tag*{\EOE}}
\def\uno{\mathds{1}}

\def\bm{\left[\begin{array}}
\def\em{\end{array}\right]}
\def\ben{\begin{enumerate}}
\def\een{\end{enumerate}}
\def\bit{\begin{itemize}}
\def\eit{\end{itemize}}
\def\barr{\begin{array}}
\def\earr{\end{array}}
\def\igdef{\ \stackrel{\mbox{\tiny{def}}}{=}\ }

\def\eps{\varepsilon}
\def\la{\lambda}

\def\N{\mathbb{N}}
\def\R{\mathbb{R}}
\def\C{\mathbb{C}}
\def\I{\mathbb{I}}

\def\T{\mathbb{T}}
\def\cC{\mathcal{C}}

\def\cH{\mathcal{H}}
\def\cK{\mathcal{K}}

\def\cT{{\cal T}}
\def\cM{{\cal M}}
\def\cB{{\cal B}}
\def\cN{{\cal N}}
\def\cV{{\cal V}}
\def\cU{{\cal U}}

\def\ca{\mathbf{a}}

\def\orto{^\perp}
\def\inc{\subseteq}

\def\inv{^{-1}}
\def\rai{^{1/2}}

\def\api{\langle}
\def\cpi{\rangle}

\def\ua{^\uparrow}
\def\da{^\downarrow}

\def\nuel{\nu(\la\coma t)}
\def\muel{\mu(\la\coma t)}

\DeclareMathOperator{\Preal}{Re} 
 \DeclareMathOperator{\tr}{tr}
\DeclareMathOperator{\gen}{span}

\DeclareMathOperator{\leqp}{\leqslant}

\DeclareMathOperator{\convf}{Conv (\R_{\ge0})}
\DeclareMathOperator{\convfs}{Conv_s (\R_{\ge0})}

\newcommand{\mat}{\mathcal{M}_d(\mathbb{C})}

\newcommand{\matsad}{\mathcal{H}(d)}

\newcommand{\matud}{\mathcal{U}(d)}

\newcommand{\matpos}{\mat^+}

\def\beq{\begin{equation}}
\def\eeq{\end{equation}}

\def\pausa{\medskip\noi}

\def\Ax2{\,( S_{E(\cF)^\#_\cV})\hat{}_x }

\newcommand{\tcal}{\T_{\C^d}(\ca)}

\newcommand{\paren}[1]{\left(#1\right)}
\newcommand{\llav}[1]{\left\{#1\right\}}

\newcommand{\norm}[1]{\left\|#1\right\|}

\newcommand{\cafo}{\cC_\ca\paren{\cF_0}}

\usepackage{color}

\definecolor{rojo}{rgb}{1,0,0}
\definecolor{azul}{rgb}{0,0,1}

\begin{document}

\title{Frame completions with prescribed norms: local minimizers and applications}
\author{ Pedro G. Massey $^{*}$, Noelia B. Rios $^{*}$ and Demetrio Stojanoff 
\footnote{Partially supported by CONICET 
(PIP 0150/14), FONCyT (PICT 1506/15) and  FCE-UNLP (11X681), Argentina. } \ 
 \footnote{ e-mail addresses: massey@mate.unlp.edu.ar, nbrios@mate.unlp.edu.ar, demetrio@mate.unlp.edu.ar}
\\ {\small Depto. de Matem\'atica, FCE-UNLP
and IAM-CONICET, Argentina  }}

\date{\small {Dedicated to the memory of  Mar\'\i a Amelia Muschietti}}
\maketitle
\begin{abstract}
Let $\cF_0=\{f_i\}_{i\in\mathbb{I}_{n_0}}$ be a finite sequence of vectors in $\C^d$ and let 
$\ca=(a_i)_{i\in\mathbb{I}_k}$ be a finite sequence of positive numbers.
We consider the completions of $\cF_0$ of the form $\cF=(\cF_0,\cG)$ obtained by appending a sequence 
$\cG=\{g_i\}_{i\in\mathbb{I}_k}$ of vectors in $\C^d$ such that $\|g_i\|^2=a_i$ for $i\in\mathbb{I}_k$, 
and endow the set of completions with the metric $d(\cF,\tilde \cF)
=\max\{ \,\|g_i-\tilde g_i\|: \ i\in\mathbb{I}_k\}$ where $\tilde \cF=(\cF_0,\,\tilde \cG)$. 
In this context we show that local minimizers on the set of completions of a convex potential 
$\text{P}_\varphi$, induced by a strictly convex function $\varphi$, 
are also global minimizers. In case that $\varphi(x)=x^2$ then $\text{P}_\varphi$ is the so-called 
frame potential introduced by Benedetto and Fickus, and our work generalizes several well known 
results for this potential. We show that there is an intimate connection between frame completion 
problems with prescribed norms and frame operator distance (FOD) problems. We use this connection 
and our results to settle in the affirmative a generalized version of Strawn's conjecture on the FOD.
\end{abstract}

\noindent  AMS subject classification: 42C15, 15A60.

\noindent Keywords: frame completions, convex potentials, 
local minimum, majorization.

\tableofcontents

%--------------------------------------------------------------------------------------- 
\section{Introduction}
%---------------------------------------------------------------------------------------

A family $\cF=\{f_i\}_{i\in\I_n}\in(\C^d)^n$ is a frame for $\C^d$ if it generates $\C^d$. 
Equivalently, $\cF$ is a frame for $\C^d$ if there exist positive constants $0<A\leq B$ such that 
\beq\label{eq cotas frame}
 A\ \|f\|^2\leq \sum_{i\in\I_n} |\langle f,f_i\rangle |^2\leq B\ \|f\|^2 \peso{for} f\in\C^d\,.
\eeq
As a (possibly redundant) set of generator, a frame $\cF$ provides linear representations of vectors in $\C^d$. 
Indeed, it is well known that in this case 
\beq\label{eq for reconstr}
 f=\sum_{i\in\I_n}\langle f,g_i\rangle \ f_i=\sum_{i\in\I_n}\langle f,f_i\rangle \ g_i
\eeq
for certain frames $\cG=\{g_i\}_{i\in\I_n}$, that are the so-called dual frames of $\cF$. Thus, a vector (signal) $f\in\C^d$ can
be encoded in terms of the coefficients $(\langle f,g_i\rangle)_{i\in\I_n}$; these coefficients can be sent (one by one) through a transmission channel 
and the receiver can then reconstruct $f$, by decoding the sequence of coefficients using the reconstruction formula in Eq. \eqref{eq for reconstr}. 

\pausa
Frames are of interest in applied situations, in which their redundancy can be used to 
deal with real-life problems, such as noise in the transmission channel (leading to what is known in the literature as robust frame designs).
The stability of the reconstruction algorithm in Eq. \eqref{eq for reconstr} also plays a central role in applications of frame theory. The consideration of these features of frames motivated the introduction of 
{\it unit norm tight} frames, which are those frames for which we can choose $A=B$ in Eq. \eqref{eq cotas frame} and such that $\|f_i\|=1$ for $i\in\I_n$.
It turns out that unit norm tight frames have several optimality properties related with erasures of the frame coefficients and numerical stability of their reconstruction formula \cite{caskov,HolPau}. 

\pausa 
In the seminal paper \cite{BF} Benedetto and Fickus gave another characterization of unit norm tight frames, in terms of a convex functional known as the frame potential. Indeed, given a finite sequence $\cF=\{f_i\}_{i\in\I_n}$ in $\C^d$ then the frame potential of $\cF$, denoted $\text{FP}(\cF)$, is given by 
\beq \label{eq defi pot frame intro}
 \text{FP}(\cF)=\sum_{i\coma j\in\I_n}|\langle f_i\coma f_j\rangle |^2\,.
\eeq
Benedetto and Fickus showed that if we endow the set of unit norm frames with $n$ elements in $\C^d$ with the metric $d(\cF,\,G)=\max\{\|f_i-g_i\|:\ i\in\I_n\}$ then unit norm tight frames are characterized as local minimizers of the frame potential, and that are actually global minimizers of this functional (among unit norm frames). 
This was the first indirect proof of the existence of unit norm tight frames (for $n\geq d$).
In applications, it is sometimes useful to consider frames $\cF=\{f_i\}_{i\in\I_n}$ with norms prescribed by a sequence $\ca=(a_i)_{i\in\I_n}$ i.e. such that $\|f_i\|^2=a_i$ for $i\in\I_n$. The consideration of these families raised the question of whether there exist tight frames with (arbitrary) prescribed norms, leading to what is known as frame design problems. It turns out that a complete solution to the frame design problem can be obtained in terms of the Schur-Horn theorem, which is a central result in matrix analysis; moreover, this characterization showed that for some sequences $\ca$ there is no tight frame with norms given by $\ca$ (see \cite{Illi,CFMP,Casagregado,CMTL,ID,DFKLOW,FWW,KLagregado}).

\pausa
The absence of tight frames in the class $\mathbb F_{\C^d}(\ca)$ of frames with norms prescribed by a fixed sequence $\ca$ lead to consider 
substitutes of tight frames within this class i.e., frames in $\mathbb F_{\C^d}(\ca)$ that had some optimality properties within this class. A complete solution of the optimal design problem with prescribed norms with respect to the frame potential was given in \cite{Phys}
where the global minimizers of $\text{FP}$ were computed; moreover, the authors obtained a crucial property of these optimal 
frame designs: they showed that if we endow $\mathbb F_{\C^d}(\ca)$ with the metric
$d(\cF,\,G)=\max\{\|f_i-g_i\|:\ i\in\I_n\}$ then local minimizers of $\text{FP}$ in $\mathbb F_{\C^d}(\ca)$ are actually 
optimal designs i.e. that local minimizers are global minimizers. This generalization of the results from \cite{BF}
motivated the study of perturbation problems related with gradient descent method of the (smooth function) $\text{FP}$
in the (smooth) manifold $\mathbb F_{\C^d}(\ca)$.

\pausa
It turns out that the frame potential can be considered within the general class of 
convex potentials introduced in \cite{mr2010}. Moreover, in \cite{mr2010} it was shown 
that the optimal frame designs in $\mathbb F_{\C^d}(\ca)$ obtained in \cite{Phys} were actually global minimizers 
of every convex potential within this class. In \cite{mrs3} the authors showed further that 
local minimizers of any convex potential
induced by a strictly convex function 
are global minimizers of every convex potential within $\mathbb F_{\C^d}(\ca)$, settling in the affirmative a conjecture from \cite{mr2010}.

\pausa
In \cite{FMP2}, given an initial sequence of vectors $\cF_0$ in $\C^d$ and a sequence of positive numbers $\ca=(a_i)_{i\in\I_k}$, the authors 
posed the problem of computing the completions $\cF=(\cF_0,\cG)$ obtained by appending a sequence $\cG=\{g_i\}_{i\in\I_k}$ in $\C^d$ with norms prescribed by $\ca$, such that these completions minimize the so-called mean squared error (MSE). 
This is known as the optimal completion problem with prescribed norms for the MSE, and contains the optimal design problem with prescribed norms for the MSE as a particular case (i.e. if $\cF_0=\emptyset$). It turns out that the MSE is also a convex potential. In the series of papers
\cite{mrs1,mrs2,mrs3} a complete solution to the optimal completion problem with prescribed norms was obtained with respect to every convex potential; explicitly, the authors showed that there exists a class of completions with prescribed norms, determined by certain spectral conditions, such that the members of this class minimize simultaneously every convex potential among such completions.
This fact was independently re-obtained in 
\cite{FMaP}, in terms of a generalized Schur-Horn theorem. 
Notice that there is a natural metric in the set of completions given by $d(\cF\coma\tilde \cF)=\max\{\|g_i-\tilde g_i\|:\ i\in\I_k\}$ for $\cF=(\cF_0,\cG)$ and $\tilde \cF=(\cF_0,\tilde \cG)$. 
Yet, the structure of local minimizers of convex potentials of frame completions with prescribed norms 
was not obtained in these works, not even for Benedetto-Fickus' frame potential.

\pausa
In (the seemingly unrelated paper) \cite{strawn}, Strawn considered an approximate gradient descent method for the {\it frame operator distance} (FOD), in the smooth manifold $\T_{\C^d}(\ca)$ of sequences with norms prescribed by $\ca$. The algorithm essentially searches for critical points of the FOD. In such cases, one expects to reach - at best - local minimizers of the objective function. It is then relevant to understand the nature of local minima, as this exhibits some aspects of the numerical performance of the algorithm. Based on computational evidence, Strawn conjectured that - under some technical assumptions - local minimizers of the FOD are also global minimizers. As a motivation for studying the FOD, the author observed in \cite{strawn} that in some cases minimization of the FOD is equivalent to minimization of the frame potential in $\T_{\C^d}(\ca)$. 

\pausa
In this paper, given 
an initial sequence of vectors $\cF_0$ in $\C^d$ and a sequence of positive numbers $\ca=(a_i)_{i\in\I_k}$,
we show that any completion $\cF=(\cF_0,\cG_0)$ of $\cF_0$  
obtained by appending a sequence in $\cG_0\in\T_{\C^d}(\ca)$ (i.e. with norms prescribed by the sequence of positive numbers $\ca$) that is a local minimizer of some 
(strictly) convex potential, is a global minimizer 
of every convex potential among such completions. Thus, our results generalize those of \cite{BF,Phys,FMP,FMP2} for the frame potential and MSE, and those from \cite{FMaP,mr2010,mrs1,mrs2,mrs3} related with optimal designs/completions with prescribed norms with respect to arbitrary convex potentials.
These results suggest the implementation of (approximate) gradient descent algorithms
for computing (optimal) solutions to frame perturbation problems. As a tool we develop 
a local version of Lidskii's additive inequality, that is of independent interest. 
We apply these results to settle in the affirmative 
Strawn's conjecture on the structure of local minimizers of the frame operator distance from \cite{strawn}. We approach this conjecture by means of a translation between frame completion problems and FOD problems. Moreover, we compute distances between certain sets of positive semidefinite matrices that generalize the FOD, in terms of arbitrary unitarily invariant norms.

\pausa
The paper is organized as follows. In Section \ref{sec prelis} we introduce the notation and terminology as well as some results from matrix analysis and frame theory used throughout the paper. We have included section \ref{sec feas} in which we summarize several results from \cite{mrs1,mrs3} for the benefit of the reader. In
Section \ref{sec 3 tutti} we show some features of local minimizers of convex potentials within the set of frame completions with prescribed norms. Our approach is based on a local Lidskii's theorem that we describe in Section \ref{sec loc lid} (we delay its proof to Section \ref{appendicity} - Appendix). We use this result in Sections \ref{sec geo loc min} and \ref{estruc interior} to obtain geometrical and spectral properties of local minimizers. In section \ref{sec loc son glob} we prove our main result, namely that local minimizers of strictly convex potentials within the set of frame completions with prescribed norms are also global minimizers. In Section \ref{sec aplicados} we apply the main result to prove (a generalized version of) Strawn's conjecture on local minima of the frame operator distance. The paper ends with Section \ref{appendicity} (Appendix) in which we show a local version of Lidskii's inequality.

%---------------------------------------------------------------------------------------
\section{Preliminaries}\label{sec prelis}
%---------------------------------------------------------------------------------------

In this section we introduce the notations, terminology and results from matrix analysis and frame theory
that we will use throughout the 
paper. General references for these results are the texts \cite{Bhat} and \cite{TaF,FinFram, Chr}.

\subsection{Preliminaries from matrix analysis}\label{sec2.2}
\pausa
In what follows we adopt the following 

\pausa
{\bf Notations and terminology}. We let $\mathcal M_{k,d}(\C)$ be the space of complex $k\times d$ matrices and write $\mathcal M_{d,d}(\C)=\mat$ for the algebra of $d\times d$ complex matrices. We denote by $\matsad\subset \mat$ the real subspace of selfadjoint matrices and by $\matpos\subset \matsad$ the cone of positive semidefinite matrices. We let $\matud\subset \mat$ denote the group of unitary matrices.
For $d\in\N$, let $\I_d=\{1,\ldots,d\}$. Given $x=(x_i)_{i\in\I_d}\in\R^d$ we denote by $x\da=(x_i\da)_{i\in\I_d}$ (respectively $x\ua=(x_i\ua)_{i\in\I_d}$) the vector obtained by rearranging the entries of $x$ in non-increasing (respectively non-decreasing) order. We denote by $(\R^d)\da=\{x\da:\ x\in\R^d\}$, $(\R_{\geq 0}^d)\da=\{x\da:\ x\in\R_{\geq 0}^d\}$ and analogously $(\R^d)\ua$ and $(\R_{\geq 0}^d)\ua$. Given a matrix $A\in\matsad$ we denote by $\la(A)=\la\da(A)=(\la_i(A))_{i\in\I_d}\in (\R^d)\da$ the eigenvalues of $A$ counting multiplicities and arranged in 
non-increasing order, and by $\la\ua(A)$ the same vector  but ordered 
in non-decreasing order. For $B\in\mat$ we let $s(B)=\la(|B|)$ denote the singular values of $B$, i.e. the eigenvalues of $|B|=(B^*B)^{1/2}\in\matpos$.
If $x,\,y\in\C^d$ we denote by $x\otimes y\in\mat$ the rank-one matrix given by $(x\otimes y) \, z= \langle z\coma y\rangle \ x$, for $z\in\C^d$.

\pausa Next we recall the notion of majorization between vectors, that will play a central role throughout our work.
\begin{fed}\rm 
Let $x\in\R^k$ and $y\in\R^d$. We say that $x$ is
{\it submajorized} by $y$, and write $x\prec_w y$,  if
$$
\suml_{i=1}^j x^\downarrow _i\leq \suml_{i=1}^j y^\downarrow _i \peso{for every} 1\leq j\leq \min\{k\coma d\}\,.
$$
If $x\prec_w y$ and $\tr x = \sum_{i=1}^kx_i=\sum_{i=1}^d y_i = \tr y$,  then $x$ is
{\it majorized} by $y$, and write $x\prec y$.
\EOE
\end{fed}
\pausa Given $x,\,y\in\R^d$ we write
$x \leqp y$ if $x_i \le y_i$ for every $i\in \mathbb I_d \,$.  It is a standard  exercise 
to show that $x\leqp y \implies x^\downarrow\leqp y^\downarrow  \implies x\prec_w y $.

\pausa 
Although majorization is not a total order in $\R^d$, there are several fundamental inequalities in 
matrix theory that can be described in terms of this relation. As an example of this phenomenon we can consider 
Lidskii's (additive) inequality (see \cite{Bhat}). In the following result we also include the characterization of the case of equality obtained in \cite{mrs2}.

\begin{teo}[Lidskii's inequality]\label{mrs284}\rm
Let $A,\,B\in \matsad$ with eigenvalues $\la(A)$, $\la(B)\in (\R^d)^\downarrow$ respectively. Then 
\ben
\item $\la\ua(A)+\la\da(B) \prec \la(A+B)$.
\item If $\la(A+B)=\paren{\la(A)+\la^{\uparrow}(B)}^{\downarrow}$ then
there exists 
$\{v_i\}_{i\in\I_d}$ an ONB of $\C^d$ such that 
\beq
A=\sum_{i\in\I_d}\la_i (A) \ v_i\otimes v_i \py B=\sum_{i\in\I_d}\la_i\ua (B)\ v_i\otimes v_i\  .
\QEDP 
\eeq\een
\end{teo}
\pausa
Recall that a norm $\nui{\cdot}$ in $\mat$ is unitarily invariant if 
$$ \nui{\,U\,A\,V\,}=\nui{A} \peso{for every} A\in\mat \py U,\,V\in\matud\,.$$
Examples of unitarily invariant norms (u.i.n.) are the spectral norm $\|\cdot\|$ and the $p$-norms $\|\cdot\|_p$, for $p\geq 1$.
It is well known that majorization is intimately related with tracial inequalities of convex functions and also with 
inequalities with respect to u.i.n's. 
The following result summarizes these relations (see for example \cite{Bhat}):

\begin{teo}\label{teo intro prelims mayo}
Let $x,\,y\in \R^d$ and let $A,\,B\in\mat$. If
$\varphi:I\rightarrow \R$ is a 
convex function defined on an interval $I\inc \R$ such that 
$x,\,y,\,s(A),\,s(B)\in I^d$ then: 
\ben 
\item If $x\prec y$, then 
$ 
\tr \varphi(x) \igdef\suml_{i\in\IN{d}}\varphi(x_i)\leq \suml_{i\in\IN{d}}\varphi(y_i)=\tr \varphi(y)\ .
$
\item If only $x\prec_w y$,  but $\varphi$ is an increasing function, then  still 
$\tr \varphi(x) \le \tr \varphi(y)$. 
\item If $x\prec y$ and $\varphi$ is a strictly convex function such 
that $\tr \,\varphi(x) =\tr \, \varphi(y)$ then there exists a permutation $\sigma$ 
of $\IN{d}$ such that $y_i=x_{\sigma(i)}$ for $i\in \IN{d}\,$, i.e. 
$x\da=y\da$. 
\item If $s(A)\prec_w s(B)$ then $\nui{A}\leq \nui{B}$, for every u.i.n. $\nui{\cdot}$ defined on $\mat$. 
\een \qed
\end{teo}

\subsection{Frames and frame completions with prescribed norms}

\pausa
In what follows we adopt the following notations and terminology from frame theory.

\pausa
{\bf Notations and terminology}: let $\cF=\{f_i\}_{i\in\I_k}$ be a finite sequence in $\C^d$. Then,
\ben
\item $T_\cF\in \cM_{d,k}(\C)$ denotes the synthesis operator of $\cF$ given by $T_\cF\cdot(\alpha_i)_{i\in\I_k}=\sum_{i\in\I_k}\alpha_i\, f_i$.
\item $T_\cF^*\in \cM_{k,d}(\C)$ denotes the analysis operator of $\cF$ and it is given by $T_\cF^*\cdot f=(\langle f,f_i\rangle)_{i\in\I_k}$.
\item  $S_\cF\in \matpos$ denotes the frame operator of $\cF$ and it is given by $S_\cF=T_\cF\,T_\cF^*$. Hence, $Sf=\sum_{i\in\I_k} \langle f,f_i\rangle f_i=\sum_{i\in\I_k} f_i\otimes f_i (f)$ for $f\in\C^d$. 
\item We say that $\cF$ is a frame for $\C^d$ if it spans $\C^d$; equivalently, $\cF$ is a frame for $\C^d$ if $S_\cF$ is a positive invertible operator acting on $\C^d$.
\een

\pausa
In several applied situations it is desired to construct 
a sequence $\cG$ in $\C^d$, in
such a way that the  frame 
operator of $\cG$ is given by some positive operator  $B\in\matpos$ and the squared norms of the frame elements 
are prescribed by a sequence of positive numbers  $\ca=(a_i)_{i\in\I_k}$. 
This is known as the classical 
frame design problem and it has been studied by several research groups (see for example \cite{Illi,CFMP,Casagregado,CMTL,ID,DFKLOW,FWW,KLagregado}). The following result 
characterizes the existence of such frame design in terms of majorization relations.

\begin{pro}[\cite{Illi,MR0}]\label{frame mayo}\rm
Let $B\in\matpos$ with eigenvalues $\lambda(B)=(\lambda_i)_{i\in\I_d}\in (\R_{\ge0}^d)\da$ and consider  
$\ca=(a_i)_{i\in\I_k}\in(\R_{>0}^k)\da$. Then there exists 
a sequence $\cG=\{g_i\}_{i\in\I_k}$ in $\C^d$ with frame operator 
$S_\cG= B$ such that $\|g_i\|^2=a_i$ for $i\in\I_k$  
if and only if $\ca\prec\la(B)$ i.e. 
\beq
\sum_{i\in\I_j} a_i\leq \sum_{i\in\I_j}\lambda_i 
 \ \ ,\  \peso{for} 1\leq j\leq \min\{k\coma d\} \qquad \text{ and } \qquad \sum_{i\in\I_k}a_i= \sum_{i\in\I_d}\lambda_i \, .
\QEDP
\eeq

\end{pro}

\pausa Recently, researchers have asked about the structure of {\it optimal frame completions with prescribed norms}. Explicitly,
let $\cF_0=\{f_i\}_{i\in\I_{n_0}}$ be a fixed (finite) sequence of vectors in $\C^d$, consider a sequence $\ca
= (a_i)_{i\in\I_k}\in(\R_{>0}^k)\da$ 
 and denote by $n=n_0+k$.  Then, 
with this fixed data, the problem is to construct a sequence  
$$
\cG = \{g_i\}_{i=1}^k \peso{with} \|g_{i}\|^2=a_i \peso{for}
i\in\I_k\ , 
$$
such that the resulting completed sequence $\cF= (\cF_0\coma \cG )$ - obtained by appending the sequence $\cG$ to $\cF_0$ - is such that the eigenvalues of the frame operator of  $\cF$ are as concentrated as possible: thus, ideally, we would search for completions $\cG$ such that $\cF=(\cF_0 \coma \cG )$ 
is a tight frame i.e. such that $S_{\cF}=c\,I$ for some $c>0$. Unfortunately, it is well known that there might not exist such completions
(see \cite{FWW,FMP,FMP2,MR0,mrs1,mrs2,mrs3}). 
We could measure optimality in terms of the frame potential i.e., 
we search for a frame 
$\cF= (\cF_0\coma \cG )$, with
$\|g_{i}\|^2=a_i$ for $1\leq i\leq k$, and such that its frame potential 
$\FP(\cF)=\tr \,S_\cF ^2 $ is minimal  
among all possible such completions; 
alternatively, we could measure optimality in terms of the so-called mean squared error (MSE) of the completed sequence $\cF$ i.e. MSE$(\cF)=\tr(S_\cF^{-1})$ (see \cite{FMP2}).  
 More generally, we can measure stability of the completed frame $\cF=(\cF_0,\cG)$ 
in terms of general convex potentials. In order to introduce these potentials we consider the sets
$$
\convf = \{ 
\varphi:\R_{\geq 0} \rightarrow \R_{\geq 0}\ :\  \varphi   \ \mbox{ is a convex function} \ \} 
$$  
and $\convfs = \{\varphi\in \convf : \varphi$ is strictly convex $\}$. 

\begin{fed}\label{pot generales}\rm
Following \cite{mr2010} we consider the 
(generalized) convex potential $\text{P}_\varphi$ associated to $\varphi\in \convf$, given by
$$
\barr{rl}
\text{P}_\varphi(\cF)&=\tr \, \varp(S_\cF) = \sum_{i\in\I_d}\varphi(\lambda_i(S_\cF)\,) \peso {for} 
\cF=\{f_i\}_{i\in\I_n}\in  (\C^d)^n \ , \earr
$$
where the matrix $\varp(S_\cF)$ is defined by means of the usual functional calculus. \EOE
\end{fed}
\pausa
Convex potentials allow us to model several well known measures of stability considered in frame theory. For example,
in case $\varp(x)=x^2$ for $x\in\R_{\geq 0}$ then $\text{P}_\varphi$ is the Benedetto-Fickus frame potential; in case 
$\varp(x)=x^{-1}$ for $x\in \R_{>0}$ then $\text{P}_\varphi$ is known as the mean squared error (MSE).

\pausa
We can now give a detailed description of the optimal completion problem with prescribed norms with respect to convex potentials.

\begin{notas}\label{rem prob com con norm predet}%\rm
Let $\cF_0=\{f_i\}_{i\in\I_{n_0}}\in (\C^d)^{n_0}$ and $\ca=(a_i)_{i\in\I_k}\in(\R_{>0}^k)\da$. 
\ben
\item 
Let $\tcal = \llav{ \cG=\{g_i\}_{i\in\I_k} \in  (\C^d)^{k} \  : 
\norm{g_{i}}^2=a_i  \mbox{ for every }   i \in \IN{k}}$
\item 
We consider the set $\cafo$ of completions of $\cF_0$ with norms prescribed by the sequence $\ca$ given by
$$
\cafo 
=\llav{\cF=  (\cF_0 \coma \cG) \in  (\C^d)^{n_0+k} \  : 
\cG \in \tcal} \,.
$$
\item For $\varphi\in\convf$ we consider the optimal frame 
completions $\cF^{\rm op}=(\cF_0,\cG^{\rm op})\in\cafo$ with respect to the convex potential P$_\varphi$ i.e. 
such that 
\beq
\text P_\varphi(\cF^{\rm op})=\min\{ \ \text P_\varphi(\cF):\  \cF\in\cafo \ \} \,. \EOEP \eeq 
\een
\end{notas}
\pausa
Consider the Notations \ref{rem prob com con norm predet} above.
In the series of papers \cite{mrs1,mrs2,mrs3} the spectral and geometrical structure of optimal frame completions $\cF^{\rm op}=(\cF_0,\cG^{\rm op})\in\cafo$  was completely described,
in case $\varphi\in\convfs$ (see Theorems \ref{teo princ jfaa1} and \ref{teo princ jfaa2} below); indeed, in this case it was shown that if $\cF^{\rm op}\in \cafo$ is optimal with respect to P$_\varphi$ in $\cafo$, then $\cF^{\rm op}$ is also optimal in $\cafo$ with respect to any other convex potential.

\subsection{Optimal frame completions with prescribed norms: feasible cases}\label{sec feas}

In this subsection we consider several concepts related with the notion of feasible index introduced in \cite{mrs3}. The feasible indexes (see Definition \ref{s feas}) 
will play a key role in our study of frame completions that are local minima of strictly convex potentials.

\begin{fed}\label{defi nuel} \rm 
Let $\la=(\la_i)_{i\in\I_d}\in (\R_{\geq 0}^d)\ua$ and fix $t\in \R_{>0}\,$. 
\ben
\item Consider the function $h_\la:[\la_1\coma \infty)\rightarrow  [0\coma \infty)$ given by 
$$
h_\la(x)=\sum_{i\in\I_d}(x-\la_i)^+ \  , \peso{for every} x\in [\la_1\coma \infty)\ ,
$$ 
where
$y^+=\max\{y,\,0\}$ stands for the positive part of $y\in\R$. 
It is easy to see that  $h_\la$ is continuous, strictly increasing, such that $h_\la(\la_1)=0$ 
and $\lim\limits_{x\rightarrow +\infty}h_\la(x)=+\infty$. 
\item Therefore $h_\la \inv :[0\coma \infty)\rightarrow [\la_1\coma \infty) $ is well defined and bijective; hence, there exists a unique 
\beq\label{el c}
c=c(t)>\la_1\geq 0 \peso{such that}  h_\la(c(t))=t>0 \ . 
\eeq
\item Let $c=c(t)>\la_1\geq 0$ be as in Eq. \eqref{el c}. %Remark \ref{rem defi c}.
Then we set \rm
\beq\label{def nuel}
\nuel \igdef \big(\ (c-\la_1)^++\la_1\coma \ldots \coma (c-\la_d)^++\la_d\,\big)\in(\R_{>0}^d)\ua\ .
\eeq
\item 
We now consider the vector 
\beq\label{def muel}
\muel  \igdef \big(\ (c-\la_1)^+\coma \ldots \coma (c-\la_d)^+\,\big) 
 = \nuel - \la  \in(\R_{\ge0}^d)\da\ .
\eeq
We remark the fact that, since $\la \in (\R_{\ge 0}^d)\ua$, 
then $\muel \in(\R_{\ge0}^d)\da$. 
\EOE
\een
\end{fed}

\begin{rem}\label{la pinta del nuel}
Let $\la=(\la_i)_{i\in\I_d}\in (\R_{\geq 0}^d)\ua$, 
let $t>0$ and let $c=c(t)>\la_1\geq 0$ be as in Eq. \eqref{el c}. %Remark \ref{rem defi c}. 
Notice that by construction, we have that 
$$
\tr \nuel=\sum_{i\in\I_d} (c-\la_i)^++\la_i=\tr \la + h_\la(c)=\tr\la + t\ .
$$ 
On the other hand, since $(c-a)^++a=\max\{c,\,a\}$ we see that:
\begin{enumerate}
\item If $ c<\la_d$ then there exists a unique $r\in\I_{d-1}$ such that, if we let $\uno_r=(1,\ldots,1)\in\R^r$ then 
\beq\label{eq nuel no constante}
\nuel=(c\,\uno_r \coma \la_{r+1}\coma \ldots\coma \la_d)\in(\R_{>0}^d)\ua \peso{with} \la_r\leq c<\la_{r+1}\,.
\eeq
In this case, $\tr\la+t=\tr \nuel< d\, \la_d$ and then $\la_d>\frac{\tr\la+t}{d}\ $.
\item Otherwise, $c\geq \la_d$ and therefore $\nuel=c\,\uno_d\in(\R_{>0}^d)\ua$. In this case 
$$
\tr\la+t=\tr\nuel=d\, c \geq d\, \la_d  \peso{and then} \la_d\leq \frac{\tr\la+t}{d}\ .
$$
\end{enumerate}
The previous remarks show that if $\rho=(e\,\uno_s \coma \la_{s+1} \coma \ldots \coma \la_d)$ 
or $\rho=e\, \uno_d$ for some $e>0$ is such that 
\beq 
\rho \in(\R_{>0}^d)\ua \ \ , \ \ 
\rho\geq \la\py \tr\rho=\tr\la+t \implies \rho=\nuel\ . \EOEP 
\eeq
\end{rem}
\pausa
Next we introduce the notion of a feasible pair. Then, we clarify the relation between this notion and the frame completion problem.

\begin{fed}\label{feasibilidad} \rm
Let $\la=(\la_i)_{i\in\I_d}\in (\R_{\geq 0}^d)\ua$ and let $\ca=(a_i)_{i\in\I_k}\in(\R_{>0}^k)\da$.
\begin{enumerate}
\item Let $r=\min\{k,d\}$, let $\tilde \la=(\la_i)_{i\in\I_r}$ and 
$ t=\sum_{i\in\I_k}a_i>0$. Let $\nu(\la\coma\ca)\in\R_{>0}^d$ be given by:
\beq\label{def nuela}
\nu(\la\coma\ca)=
\begin{cases}  \quad \quad \quad \quad
 \nu( \la\coma t) & \peso{if} \ r=d\le k \\%\peso{or} \nu(\la\coma\ca)=
\big(\, \nu(\tilde\la\coma t) \coma \la_{r+1} \coma \ldots \coma \la_d\, \big) 
& \peso{if}\ r=k<d\ .
\end{cases}  
\eeq
Observe that in the second case  $\nu(\la\coma\ca)$ could be a non ordered vector (if $c(t)>\la_{r+1}$). 
\item Consider the vector $\mu(\la\coma\ca) 
\igdef \nu(\la\coma\ca)-\la \in \R_{\geq 0}^d\,$.
By inspection of Definition \ref{defi nuel} and item 1 above 
we see that $\mu(\la\coma\ca)=\mu(\la\coma\ca)\da$ 
and $\tr\,\mu(\la\coma\ca)=\tr\,\ca = t$.

\item We say that the pair $(\la\coma\ca)$ is {\bf feasible} if $\ca\prec \mu(\la\coma\ca)$ that is, if
\beq\label{desi mayo feas}
\sum_{i\in\I_j}a_i\leq \sum_{i\in\I_j} \mu_i(\la\coma\ca)\peso{for} j\in\I_{r-1}\ , 
\eeq
where the equivalence follows from the properties of $\mu(\la\coma\ca) $ 
given in item 2. Notice that in the case that  $k<d$ then $\mu_{k+1}(\la\coma\ca) =0$.\EOE
\end{enumerate}
\end{fed}

\pausa
We point out that the computation of $\nu(\la\coma\ca)$ and $\mu (\la\coma\ca)$ as in Definition \ref{feasibilidad}, as well as the verification of the inequalities in Eq. \eqref{desi mayo feas} above can be implemented by a fast algorithm.

\pausa
The following result is taken from \cite{mrs3} (see also \cite{mrs1}) and it describes, 
in the feasible case,  the spectral structure of global minimizers of the convex potentials $\text P_\varphi$ in $\cafo$, for 
$\varphi\in\convfs$. As already mentioned, this structure does not depend on $\varphi$. 

\begin{teo}\label{pre teo caso feasible} \rm 
Let $\cF_0=\{f_i\}_{i\in\I_{n_0}}\in(\C^d)^{n_0}$ and let $\ca=(a_i)_{i\in\I_k}\in(\R_{>0}^k)\da$. Let $\la=\la(S_{\cF_0})\ua$ and assume that the pair $(\la\coma\ca)$ is feasible. Let $\nu(\la\coma\ca)=(\nu_i(\la\coma\ca))_{i\in\I_d}\in\R_{\geq 0}^d$ be
as in Definition \ref{feasibilidad}. Then, for every $\varphi\in\convfs$ we have that 
$$ 
\min\{\text P_\varphi(\cF): \ \cF=(\cF_0\coma\cG)\in\cafo\}
=\sum_{i\in\I_d}\varphi(\nu_i(\la\coma\ca)\,)\ .
$$ Moreover, given $\cF=(\cF_0\coma\cG)\in\cafo$ then 
\beq
\text P_\varphi(\cF)=\sum_{i\in\I_d}\varphi(\nu_i(\la\coma\ca))\iff 
 \la(S_{\cF})=\nu(\la\coma\ca)\da\ . \QEDP \eeq
\end{teo}

\section{Local minima of frame completions with prescribed norms}\label{sec 3 tutti}

We begin with a brief description of our main problem (for a detailed description see Section \ref{sec geo loc min}). 
Let $\cF_0=\{f_i\}_{i\in\I_{n_0}}$ be a fixed family in 
$(\C^d)^{n_0}$ and $\ca=(a_i)_{i\in\I_k}\in(\R_{>0}^k)\da$; consider $\T_{\C^d}(\ca)$ (see Notations \ref{rem prob com con norm predet}) endowed with the $d(\cG\coma \tilde\cG)=\max\, \{ \, \| g_i-\tilde g_i\|:\ i\in\I_k\, \}$ for
$\cG=\{g_i\}_{i\in\I_k},\, \tilde \cG=\{\tilde g_i\}_{i\in\I_k}\in \T_{\C^d}(\ca)$. 
Our main goal is to study the structure of local minimizers of the map
\beq\label{eq defi intro sec3}
 \T_{\C^d}(\ca)\ni \cG \mapsto \text{P}_\varphi(\cF_0\coma\cG)=\tr(\varphi(S_{(\cF_0\coma \cG)}))=\tr(\varphi(S_{\cF_0}+S_\cG))\,,
\eeq
where $\varphi\in\convfs$ is a strictly convex function and $\cF=(\cF_0\coma\cG)\in\cafo$ is a completion of $\cF_0$ with a sequence of vectors in $\C^d$ with norms prescribed by the sequence $\ca$. 

\pausa
In this section we describe the first structural features of local minimizers of the map in Eq. \eqref{eq defi intro sec3}, for  
general strictly convex potentials. These results are applied in the next section to prove that local minima are also global minima. 

\subsection{On a local Lisdkii's theorem}\label{sec loc lid}

\pausa
The result in this subsection lies in the context of matrix analysis, and it is of independent interest. It will be systematically used in the rest of the paper. 
Since its proof is rather technical, we shall present it in the Appendix (see Section  \ref{appendicity}). In order to put our result in perspective, we consider the following
\begin{rem}[Lidskii's (global) inequality]\label{rem Lidskii glob}

Fix  $S\in\matpos$ and $\mu\in(\R_{\geq 0}^d)\da$. Consider the unitary orbit 
\beq\label{defi omu}
\cO_\mu=\{G\in\matpos:\la(G)=\mu\}=\{U^* D_{\mu}\,U:\ U\in\matud\} 
\eeq
where $D_\mu\in\matpos$ denotes the diagonal matrix with main diagonal $\mu$.  

\pausa
Given $\varphi\in\convfs$ we define the map 
$\Phi=\Phi_{S,\,\varphi}:\cO_\mu\rightarrow \R_{\geq 0}$ given by 
\beq\label{defi Phi}
 \Phi(G)=\tr(\varphi(S+G))=\sum_{j\in\I_d}\varphi(\la_j(S+G))\peso{for} G\in\cO_\mu\,.
\eeq
(Notice that Eq. \eqref{eq defi intro sec3} motivates the consideration of the map $\Phi$ as defined above). 
Let $(\la_i)_{i\in\I_d}=\la(S)\ua\in(\R_{\geq 0}^d)\ua$. Then, Lidskii's 
(additive) inequality together with the characterization of the 
case of equality given in Theorem \ref{mrs284} imply that 
$$
 \min_{G\in\cO_\mu} \Phi(G)=\sum_{i\in\I_{d}}\varphi(\la_i+\mu_i) 
$$
and, if $G_0\in  \cO_\mu$ then $G_0$ is a global minimum of $\Phi$ on $\cO_\mu$ if and only if 
there exists $\{v_i\}_{i\in\I_d}$ an ONB of $\C^d$ such that 
$$
S=\sum_{i\in\I_d}\la_i\ v_i\otimes v_i \py G_0=\sum_{i\in\I_d}\mu_i\ v_i\otimes v_i\,.
$$
Indeed, Lidskii's inequality states that $\la(S)\ua+\la(G)\prec\la(S+G)$ for every $G\in\cO_\mu$; since $\varp\in\convfs$, the previous majorization relation implies that 
$$ \Phi(G)=\tr(\varphi(S+G))\geq \sum_{i\in\I_{d}}\varphi(\la_i+\mu_i)\,.$$ If we assume 
that $G_0\in\cO_\mu$ is a global minimum then (see Theorem \ref{teo intro prelims mayo}), $$ \la(S)\ua+\la(G_0)\prec\la(S+G_0) \py \tr(\varphi(\la(S)\ua+\la(G_0)))=\tr(\varphi(\la(S+G_0)))$$
$$\implies \la(S+G_0)=(\la(S)\ua+\la(G_0))\da\,,$$
so equality holds in Lidskii's inequality. Hence we can apply Theorem \ref{mrs284}. 
Notice that in particular, $S$ and $G_0$ commute.\EOE
\end{rem}

\pausa
Let $\mu\in(\R_{\geq 0}^d)\da$ and consider the unitary orbit $\cO_\mu$ from Eq.
\eqref{defi omu}. In what follows we consider $\cO_\mu$ endowed with the metric induced by the operator norm. 
The next result states that given $\varp \in \convfs$ then the {\bf local} minimizers of the map
$\Phi=\Phi_{S,\,\varphi}:\cO_\mu\rightarrow \R_{\geq 0}$ given by 
Eq. \eqref{defi Phi} - in the metric space $\cO_\mu$ - are also global minimizers.

\begin{teo}[Local Lidskii's theorem]\label{teo LLT}
Let $S\in\matpos$ and $\mu=(\mu_i)_{i\in\I_d}\in(\R_{\geq 0}^d)\da$. Assume that 
$\varp \in \convfs$ and that $G_0\in\cO_\mu$ is a local minimizer of $\Phi=\Phi_{S\coma \varp}$ on $\cO_\mu\,$. 
Then, there exists $\{v_i\}_{i\in\I_d}$ an ONB of $\C^d$ such that, 
if we let $(\la_i)_{i\in\I_d}=\la\ua(S)\in (\R_{\geq 0}^d)\ua$ then 
\beq\label{eq base buena}
 S=\sum_{i\in\I_d}\la_i\ v_i\otimes v_i \py G_0=\sum_{i\in\I_d}\mu_i\ v_i\otimes v_i\ .
\eeq
 In particular, $\la(S+G_0)=(\la(S)\ua+\la(G_0)\da)\da$ so $G_0$ is also a global minimizer of $\Phi$ on $\cO_\mu\,$.
\end{teo}
\proof See Theorem \ref{teo LLTApp} in the Appendix. \QED

\subsection{Geometrical properties of local minima}\label{sec geo loc min}
In the following two sections we study the relative geometry (of both the frame vectors and the eigenvectors of their frame operator) of frame completions with prescribed norms that are local minima of strictly convex potentials. 
We begin by introducing the basic notations 
used throughout these sections. 

\begin{notas}\label{not gen min loc} \ 

\begin{enumerate}
\item Let $\cF_0=\{f_i\}_{i\in\I_{n_0}}$ be a fixed family in 
$(\C^d)^{n_0}$ and $\ca=(a_i)_{i\in\I_k}\in(\R_{>0}^k)\da$;  
\item Given a subspace $V\subseteq\C^d$ we denote by 
$$\T_{V}(\ca)
= \llav{\cG=\{g_i\}_{i\in\I_k}\in V^k:\ \|g_i\|^2=a_i\, , \ i\in\I_k \ }$$ 
endowed with the product topology - of the usual topology
in $\T_{V}(a_i)$ for $i\in\I_k$ - i.e. induced by the metric
$$
d(\cG\coma \tilde\cG)=\max\, \{ \, \| g_i-\tilde g_i\|:\ i\in\I_k\, \}\,. $$  
\item Given $\varphi\in\convfs$, let 
$\Psi_\varphi = \Psi_{\varphi\coma \cF_0}:\tcal \rightarrow [0,\infty)$ be given by 
\beq \label{psifi}
\Psi_\varphi(\cG)=\rm P_\varphi (\cF_0 \coma \cG)=\tr\, \big(\,\varphi(S_{\cF_0}+S_\cG)\, \big)
 \peso{for every} \cG \in \tcal  ,
\eeq
the convex potential induced by $\varphi$ of the completed 
sequence $\cF=(\cF_0\coma\cG)\in\cafo$.
\item We shall fix $\cG_0=\{g_i\}_{i\in\I_k}\in \tcal $ which is 
a {\bf local minimum} of $\Psi_\varphi$ in $\tcal $. \EOE
\end{enumerate}
\end{notas}

\pausa
In the following sections we shall see  that 
$\cG_0\,$ is actually a global minimum
in $\tcal $ for $\Psi_\vfi\,$ or, in other words, that $\cF=(\cF_0\coma \cG_0)$ is a global minimum in $\cafo$ for the convex potential $\text{P}_\vfi\,$. 

\pausa 
The following result, which is based on the Local Lidskii's Theorem \ref{teo LLT}, depicts the first structural properties of local minimizers of strictly convex potentials. 
\begin{teo}\label{teo estruc 1} Consider the Notations \ref{not gen min loc}. %and assume that $\vfi\in \LLM$.
 Then,
\begin{enumerate}
\item For every $j\in\I_k$, $S_\cF=S_{\cF_0}+S_{\cG_0}$  commutes with $g_j\otimes g_j$  or equivalently,
$g_j$ is an eigenvector of $S_\cF=S_{\cF_0}+S_{\cG_0}$. 

\item There exists $\{v_i\}_{i\in\I_d}$ an ONB of $\C^d$ such that 
$$
S_{\cF_0}=\sum_{i\in\I_d} \la\ua_i(S_{\cF_0}) \ v_i\otimes v_i
\py  S_{\cG_0}=\sum_{i\in\I_d} \la\da_i(S_{\cG_0}) \ v_i\otimes v_i \ 
\ .
$$ 
In particular, we have that $\la(S_{\cF_0}+S_{\cG_0})=\big[\la\ua(S_{\cF_0})+\la(S_{\cG_0})\,\big]\da$.
\end{enumerate}
\end{teo}
 
\begin{proof}
For 
each $j\in\I_k\,$, let $S_j=S_{\cF_0}+\sum_{i\in\I_n\setminus\{j\}}g_i\otimes g_i\in\matpos$
and $\mu_{[j]}= a_j \, e_1\in\R^d_{\geq 0}\,$. Notice that, as in Eq.\eqref{defi omu},  the orbit 
$\cO_{\mu_{[j]}} = \{g\otimes g : \|g\|^2 = a_j\}$. By hypothesis, it is clear 
(comparing the maps $\Psi_\varphi$ and $\Phi_{S_j \coma \varphi}$) that the matrix $G_j=g_j\otimes g_j$ is a local minimum 
for the map $\Phi_{S_j \coma \varphi}$ on $\cO_{\mu_{[j]}}\,$. Using Theorem \ref{teo LLT}, we conclude that $S_j$ and $G_j$ commute, which 
implies item 1. 

\pausa
By hypothesis, there exists $\varepsilon>0$ such that every $U \in B_{(I\coma \eps)}\igdef  
\{U\in\cU(d): \|I-U\|<\epsilon\}$ satisfies that    $U\cdot\cG_0=\{U\,g_i\}_{i\in\I_k}\in\tcal $, 
$S_{U\cdot\cG_0}=U\,S_{\cG_0}\,U^*$ and  that $U\cdot\cG_0$ 
is close enough to $\cG_0$ so that   
$$ 
\Phi_{S_{\cF_0}\coma \varphi}(U\,S_{\cG_0}\,U^*)=\tr\, \big(\,\varphi(S_{\cF_0}+U\,S_{\cG_0}\,U^*)\, \big)=\Psi_\varphi (U\cdot\cG_0)\geq \Psi_\varphi (\cG_0) =\Phi_{S_{\cF_0}\coma \varphi}(S_{\cG_0})\ . 
$$ 
Let $\mu = \la (S_{\cG_0})\in (\R_{\geq 0}^d)\da$. 
Notice that  the map 	$\pi : \matud \to \cO_\mu$ given by $\pi(U) = U\,S_{\cG_0}\,U^*$ is open 
%that has continuous local cross sections at $I\in \matud$ 
(see \cite[Thm 4.1]{AS}), so that $\pi (B_{(I\coma \eps)})$ is an open 
neighborhood of $S_{\cG_0}$ in $\cO_\mu\,$, and 
$S_{\cG_0}$ is a local minimum for the map $\Phi_{S_{\cF_0}  \coma \varphi}$ on $\cO_{\mu}\,$. 
Item 2 now follows from
Theorem \ref{teo LLT}.
\end{proof}

\def\nug0{\nu(\cG_0)}

\begin{notas}\label{notas 2} Consider 
the Notations \ref{not gen min loc}. 
Then, Theorem \ref{teo estruc 1} allows us 
to introduce the following notions and notations:
\begin{enumerate}
\item We denote by  $\la=(\la_i)_{i\in\I_d} = \la\ua_i(S_{\cF_0}) \in (\R_{\geq 0}^d)\ua$  and 
$\mu=(\mu_i)_{i\in\I_d} = \la\da_i(S_{\cG_0}) \in (\R_{\geq 0}^d)\da$.
\item We fix $\cB=\{v_i\}_{i\in\I_d}$ an ONB of $\C^d$ as in Theorem \ref{teo estruc 1}. Hence, 
\beq\label{la y mu}
S_{\cF_0}=\sum_{i\in\I_d} \la_i \ v_i\otimes v_i
\py S_{\cG_0}=\sum_{i\in\I_d} \mu_i \ v_i\otimes v_i  \ ,
\eeq 
\item We denote by $  \nug0 = \la + \mu \in \R_{\geq 0}^d$ 
so that $S_{\cF}=\sum_{i\in\I_d} \nu_i(\cG_0) \ v_i\otimes v_i\,$. Notice that $\nug0$ is constructed by pairing the entries 
of ordered vectors (since $\la=\la\ua$ and $\mu=\mu\da$) but  $\nug0$ is not necessarily an ordered vector. Nevertheless, we have that $\la(S_\cF) = \nug0\da$. 
In what follows we obtain some properties of (the unordered vector) $\nug0$ 
\item Let $ s_\cF= \max \, \{i\in\I_d
: \mu_i \neq 0\} = \rk \, S_{\cG_0}\,$. Denote by  $W= R(S_{\cG_0})$, which reduces $S_\cF\,$.  
\item Let $S = S_\cF\big|_W \in L(W)\,$ and $\sigma(S) = \{ c_1 \coma \dots \coma c_p\}$
(where $c_1 > c_2 > \dots > c_p>0$). 
\item For each $j\in \I_p\,$, we consider the following sets of indexes: 
$$K_j = \{ i \in \I_{s_\cF} :   \nu_i(\cG_0)=\la_i +\mu_i = c_j\}  
\py J_j = \{i\in \I_k: S\,g_i = c_j \, g_i\}  \ .
$$ 
Theorem \ref{teo estruc 1} assures that   
$\I_{s_\cF}  
= \bigsqcup\limits_{j\in \I_p} 
\, K_j \py \I_k  
= \bigsqcup\limits_{j\in \I_p} 
\, J_j  \ . 
$
\item Since $R(S_{\cG_0})= \gen\{g_i : i \in \I_k\} = W = 
\bigoplus_{i\in\I_p} \ker \,(S-c_i\,I_W\,)$
then, for every $ j\in \I_p\ $, 
\beq\label{cajas}
W_j \igdef \gen\{g_i : i \in J_j\} = \ker \,(S-c_j\,I_W\,) = \gen\{v_i : i \in K_j\} \ ,
\eeq
because $g_i \in  \ker \,(S-c_j\,I_W\,)$ for every $i \in J_j\,$. Note that, by Theorem \ref{teo estruc 1}, each $W_j$ reduces both $S_{\cF_0}$ and $S_{\cG_0}\,$. \EOE
\end{enumerate}
\end{notas}

\pausa
The next remark allow us to consider reduction arguments when computing different aspects of the structure of local minima
of the completion problem with prescribed norms.

\begin{rem}[Two reduction arguments for local minima] \label{induc}
Consider the data, assumptions  and terminology fixed 
in the Notations \ref{not gen min loc} and \ref{notas 2}. 

\pausa
$a$) For any $j \leq p-1$ denote by 
$$
I_j = \I_d 
\setminus \bigcup_{i \le j} \,K_i  \ \coma  \ L_j = 
\I_k 
\setminus \bigcup_{i \le j} \,J_i  \ \coma \ 
\la^{I_j} = (\la_i)_{i \in I_j}  \ \coma  \ 
\cG_0^{(j)} = \{g_i\}_{i \in L_j}  \coma \ca^{L_j}= (a_i )_{i \in L_j} 
$$
and take some sequence $\cF_0^{(j)} $ in 
$\cH_j = \big[\, \bigoplus_{i\le j} W_i \big]\orto$
 such that $S_{\cF_0^{(j)} } = S_{\cF_0}|_{\cH_j}$ (notice that, by construction, $\cH_j$ reduces $S_{\cF_0}$). 
Then, it is straightforward to show that 
$\cG_0^{(j)}$ is a local minimizer of 
$\Psi^j_{\varphi \coma \cF_0^{(j)}}:\T_{\cH_j}(\ca^{L_j})\rightarrow \R_{\geq 0}\,$. 
Indeed, if $\cM_j$ is any sequence of $|L_j|$ vectors in $\cH_j$ with norms prescribed by $\ca^{L_j}$ then 
$\cM=(\{g_i\}_{i\in\I_k\setminus L_j} \coma \cM_j)\in \tcal $  (in some order) and %1we have that
$$
\Psi_\varphi(\cF_0\coma \cM)=\sum_{i=1}^j \varphi(c_i)\ \dim W_i + \Psi^j_\varphi(\cM_j)\ ,
$$
where the last equation is a consequence of the orthogonality relations between the families 
$\{g_i\}_{i\in\I_n\setminus L_j}$ and $\cM_j\,$. 
Also notice that the distance between $\cM_j$ and $\cG_0^{(j)}$ is the same as the distance between 
$\cM$ and $\cG_0\,$.

\pausa
The importance of the previous remark lies in the fact that 
it provides a reduction method to compute the structure 
of the sets $\cG_0^{(i)} \coma K_i $ and $J_i$ for $1\leq i\leq p$, as well as 
the set of constants $c_1>\ldots>c_p\geq 0$. Indeed, assume that we are 
able to describe the sets $\cG_0^{(1)}\coma K_1\coma J_1$ and the constant 
$c_1$ in some structural sense, using the fact that these sets are 
extremal (e.g. these sets are built on $c_1>c_j$ for $2\leq j\leq p$). Then, we can the apply the same argument to compute, for example, 
the sets $\cG_0^{(2)}\coma K_2\coma J_2$ using that these are extremal for the reduced problem described above for $j=1$. 

\pausa
$b$)  Assume that $k\in\I_{d-1}$ and let $\cF=(\cF_0,\cG_0)\in\cafo$. 
Fix a sequence   $\tilde \cF_0=\{\tilde f_i\}_{i\in\I_{n_0}}$ 
 in $W=R(S_{\cG_0})$ such that 
$S_{\tilde \cF_0}=S_{\cF_0}|_W\, $. Then, 
for every $\cM \in \T_W(\ca)$, 
$$ \Psi_{\varphi\coma  \cF_0} (\cM)=\text P_\varphi(\cF_0\coma \cM)
=\text{P}_\varphi(\tilde \cF_0\coma \cM)
+\sum_{i=k+1}^d \varphi(\la_i) = \Psi_{\varphi\coma \tilde \cF_0} (\cM)
+\sum_{i=k+1}^d \varphi(\la_i)\ , 
$$
even for $\cM= \cG_0\,$. 
The identity above shows that $\cG_0$ is a local minimizer of $\Psi_{\varphi\coma \tilde \cF_0}$ in 
$\T_W(\ca)$. 
In this setting we have that  $d\,'=\dim W=\rk(S_{\cG_0})\leq k$. So that in order to compute the structure of $\cG_0$ we can assume, as we sometimes do, that $k\geq d$. 
 \EOE
\end{rem}

\subsection{Inner structure of local minima}\label{estruc interior}

Throughout this section we consider the Notations \ref{not gen min loc} and \ref{notas 2}. Recall  that we have fixed 
$\varphi\in\LLM$ and a sequence 
$\cG_0=\{g_i\}_{i\in\I_k}\in \tcal $ which is 
a {local minimum} of the potential $\Psi_{\varphi\coma \cF_0}$ in $\tcal $.
The following result is inspired on some ideas from \cite {Phys}.
	
\begin{pro}\label{teo gen 4.6}
Let $\cF=(\cF_0\coma \cG_0)\in\cafo$ 
be as in Notations \ref{notas 2} and assume that there exist 
$j\in\I_ p$ and $c\in\sigma(S_\cF)$ such that 
$c<c_j\,$. Then, the family $\{g_i\}_{i\in J_j}$ is linearly independent.
\end{pro}
\begin{proof}
Suppose that for some $j\in\I_p$ the family $\{g_i\}_{i\in J_j}$ is linearly dependent.
Hence there exist coefficients $z_l\in \C$, $l\in J_j$ (not all 
zero) such that every $|z_l|\leq 1/2$ and  
\begin{equation}\label{eq1}
\sum_{l\in J_j}\overline{z_l} \ a_l\rai \ g_l=0\ . 
\end{equation}
Let $I_j \inc J_j$ be given by $I_j =\{l\in J_j:\ z_l\neq 0\}$. Assume that there exists 
$c\in\sigma(S_\cF)$ such that $c<c_j$
and let $h\in \C^d$ be such that $\|h\|=1$ and  $S_\cF h=c \,h$. 
For $t\in (-1/2,1/2)$ let $\cF(t)=(\cF_0\coma \cG(t))$ where  
$\cG(t)=\{g_i(t)\}_{i\in\IN{k}}$ is given by 
$$
g_l(t) = \begin{cases}  \ (1-t^2\,|z_l|^2)^{1/2} g_l+t\,z_l\,a_l\rai \, h 
& \mbox{if} \ \ l\in I_j\,;  \\%&\\
\quad \quad\quad g_l & \mbox{if} \ \ l\in K\setminus I_j  \  .
\end{cases}  
$$ Notice that $\cG(t)\in \tcal $ for $t\in (-1/2,1/2)$.
 Let $\Preal(A)= \frac {A+A^*}{2}$ denote the real part of 
$A \in \mat$. For $l\in I_j$ then 
$$
g_l(t)\otimes g_l(t)
=(1-t^2\,|z_l|^2)\ g_l\otimes g_l+ t^2\,|z_l|^2\,a_l \ h\otimes h 
+ 2\,(1-t^2\,|z_l|^2)^{1/2}\,t \ \Preal(h\otimes \overline{z_l}\, a_l\rai\, g_l)
$$
Let $S(t)$ denote the frame operator of $\cF(t)=(\cF_0\coma \cG(t))\in\cafo$, so that $S(0)=S_\cF\,$. 
Note that 
$$
S(t)=S_\cF+t^2 \sum_{l\in I_j }  |z_l|^2 \left( - g_l\otimes g_l + a_l \ h\otimes h \right) + R(t)
$$ 
where $R(t)=2 \suml_{l\in I_j }(1-t^2\,|z_l|^2)^{1/2}\,t \ \Preal(h\otimes a_l\rai\,\overline{z_l}\, g_l)$. 
Then $R(t)$ is a smooth function such that 
$$
R(0) = 0 \ \ , \ \ 
R'(0)=\sum_{l\in I_j } \Preal(h\otimes \overline{z_l}\, a_l\rai\,g_l)
=\Preal(h\otimes \sum_{l\in I_j } \overline{z_l}\,a_l\rai\, g_l) \stackrel{\eqref{eq1}}{=} 0 \ ,
$$ 
and such that $R''(0)=0$. Therefore 
$\lim\limits_{t\rightarrow 0} \ t^{-2}\ R(t)=0 $.
We now consider 
$$
V=\gen\,\big(\,\{g_l:\ l\in I_j \}\cup \{h\}\,\big)
=\gen\,\big\{\,g_l:\ l\in I_j \,\big\}\stackrel{\perp}{\oplus} \C\cdot h\ .
$$
Then $\dim V=s+1$,  for $s=\dim\gen\{g_l:\ l\in I_j \}\geq 1$. 
By construction, the subspace $V$ reduces $S_\cF$ and $S(t)$ for 
$t\in\R$, in such a way that $S(t)|_{V^\perp}=S_\cF|_{V^\perp}$ 
for $t\in \R$. On the other hand 
\beq\label{adet}
S(t)|_{V}=S_\cF|_V+t^2 \sum_{l\in I_j }  |z_l|^2 \left( - g_l\otimes g_l 
+ a_l \ h\otimes h \right) + R(t) = A(t)+R(t)\in L(V)\ ,
\eeq
where we use the fact that the ranges of the selfadjoint operators 
in the second and third term in the formula above clearly lie in $V$. 
Then $\la\big(\,S_\cF|_V\,\big)=\big(\,c_j \, \uno_s\coma   c\, \big)
\in (\R^{s+1}_{>0})\da $ and 
$$\barr{rl}
\la \Big(\, \sum_{l\in I_j }  |z_l|^2  g_l\otimes g_l\,\Big) &
=(\gamma_1 \coma \ldots \coma \gamma_s \coma 0)
\in (\R^{s+1}_{\geq 0})\da \peso{with} \gamma_s>0 \ ,\earr
$$
where we have used the definition of $s$ and the fact that $|z_l|>0$ for $l\in I_j \,$ (and the known fact that 
if $S\coma T\in \matpos \implies R(S+T) = R(S)+R(T)\,$). 
Hence, for sufficiently small $t$, 
the spectrum of the operator  $A(t)\in L(V)$
defined in \eqref{adet} is 
$$\barr{rl}
\la\big(\, A(t)\,\big) 
&=\big(\, c_j-t^2\,\gamma_s \coma \ldots \coma c_j-t^2 \,\gamma_1 \coma c
+t^2 \, \sum_{l\in I_j }a_l\,|z_l|^2 \,\big) \in (\R^{s+1}_{\geq 0})\da \ , \earr	
$$ 
where we have used the fact that $\langle g_l \coma h\rangle=0$ for every $l\in I_j \,$. 
Let us now consider 
$$
\la\big(\, R(t)\,\big)
=\big(\,\delta_1(t) \coma \ldots \coma \delta_{s+1}(t)\, \big) 
\in (\R^{s+1}_{\geq 0})\da\peso{for} t\in \R \ .
$$ 
Recall that in this case $\lim\limits_{t\rightarrow 0}t^{-2} \delta_j(t)=0$ 
for $1\leq j\leq s+1$. Using Weyl's inequality 
on Eq. \eqref{adet},  we now see that 
$\lambda \big(\,S(t)|_V\,\big)\prec \la\big(\, A(t)\,\big)
+\la\big(\, R(t)\,\big)\igdef \rho(t)\in (\R^{s+1}_{\geq 0})\da$. We know that 
$$
\barr{rl}
\rho(t)&= \big(\, c_j-t^2\,\gamma_s+\delta_1(t) \coma 
\ldots \coma c_j-t^2 \,\gamma_1+\delta_s(t) \coma 
c+t^2 \, \sum_{l\in I_j }a_l\,|z_l|^2  +\delta_{s+1}(t)\, \big) 
\\ &\\
&=  
\Big(\,c_j-t^2\,(\gamma_s-\frac{\delta_1(t)}{t^2}) \coma \ldots \coma 
c_j-t^2 \,(\gamma_1-\frac{\delta_s(t)}{t^2}) \coma 
c+t^2 \, (\sum_{l\in I_j }a_l\,|z_l|^2+\frac{\delta_{s+1}(t)}{t^2})\,\Big)  \ .
\earr
$$
Since by hypothesis $c_j>c$ then, the previous remarks show that there exists $\varepsilon>0$ such that 
if $t\in(0,\varepsilon)$ then, for every $i\in \I_s$ 
$$c_j>c_j-t^2(\gamma_{s-i+1}-\frac{\delta_i(t)}{t^2})> c+t^2 \, (\sum_{l\in I_j }a_l\,|z_l|^2+\frac{\delta_{s+1}(t)}{t^2})\,.$$
The previous facts show that for $t\in(0,\varepsilon)$ then
 $\rho(t) \prec 
\lambda(S_\cF|_V)=\big(\,c_j \, \uno_s\coma   c\, \big)$ strictly. 
Since $\varphi$ is strictly convex,  
for every $t\in(0,\varepsilon)$ we have that  
$$
\Psi_\varphi\big(\,\cG(t)\,\big)\leq 
\tr \, \varphi\big(\,\la(S_\cF|_{V^\perp})\,\big) +\tr \,\varphi \big(\,\rho(t)\,\big)
< \tr \, \varphi\big(\,\la(S_\cF|_{V^\perp})\,\big) +\tr \,\varphi \big(\,\la(\,S_\cF|_{V}\,)\,\big) 
= \Psi_\varphi(\cG_0)  \ .
$$
This last fact contradicts 
the assumption that $\cG_0$ is a local minimizer of $\Psi_\varphi$ in $\tcal $.
\end{proof}

\pausa
Recall that, according to Notations \ref{notas 2}, $c_1>\ldots>c_p\,$. Thus, 
the following result is an immediate consequence of Proposition \ref{teo gen 4.6} above. 

\begin{cor}\label{coro son li} \rm
Let $\cF=(\cF_0\coma \cG_0)\in\cafo$ be as in Notations \ref{notas 2} and assume that $p>1$. Then, the family $\{g_i\}_{i\in J_j}$ is linearly independent for every $j\in\I_{p-1}$. In particular, by Eq. \eqref{cajas}, 
\beq
\dim (W_j)=|K_j| =|J_j|  \peso{for} j\in\I_{p-1}\ . \QEDP
\eeq

\end{cor}

\begin{cor}\label{lema caso particular} Consider the Notations \ref{notas 2} and assume that $k\geq d$. If 
$\cF=(\cF_0\coma \cG_0)\in\cafo$ satisfies that $p=1$ 
(i.e., if we let $W=R(S_{\cG_0})$ then $S_\cF|_W=c_1\,P_W\in L(W)$\,). Then:
\begin{enumerate}
  \item $(\lambda,\,\bf a)$ is feasible;
	\item $\la\ua(S_\cF)=\nu(\la\coma\ca)$ and $\la(S_{\cG_0})=\mu(\la\coma\ca)$ (see Definition \ref{feasibilidad});
\item $\cF$ is a {global minimum} of $\Psi_\varphi$ in $\tcal $.
\end{enumerate}
\end{cor}
\begin{proof}
Assume first that $s_\cF=d$, i.e. that $W=\C^d$. In this case $S_\cF=S_{\cF_0}+S_{\cG_0}=c_1\,I$ and $\cF$ is a tight frame. Using the comments at the end of Remark \ref{la pinta del nuel} and Definition \ref{feasibilidad} (notice that in this case $\min\{d,k\}=d$) we see that $\nu(\la\coma\ca)=c_1\,\uno_d\,$. Hence, 
$$
\mu(\la\coma\ca)=c_1\,\uno_d-\la=\la\da(S_{\cG_0})\ .
$$ 
Since $S_{\cG_0}$ is the frame operator of $\cG_0\in\tcal $, Proposition \ref{frame mayo} shows that the majorization relation $\ca\prec\mu(\la\coma\ca)$ holds,  so that the pair $(\la\coma\ca)$ is feasible. The fact that $\cG_0$ is a global minimizer of $\Psi_\varphi$ in 
$\tcal $ now follows from  Theorem \ref{pre teo caso feasible} (or directly, being a tight completion).

\pausa
We now consider the case $s_\cF<d$. Hence, $\mu_i>0$ for $1\leq i\leq s_\cF$ and 
$$ S_\cF=\sum_{i\in\I_d} (\la_i+\mu_i)\ v_i\otimes v_i=\sum_{i\in\I_{s_\cF}} c_1\ v_i\otimes v_i + \sum_{i=s_\cF+1}^d \la_i\ v_i\otimes v_i \,.$$
In particular, $c_1=\la_{s_\cF}+\mu_{s_\cF}>\la_{s_\cF}\,$. On the other hand, $k\geq d>\dim W$, and thus $\{g_i\}_{i\in\I_k}=\{g_i\}_{i\in J_1}$ is a linearly dependent family. Hence, Proposition \ref{teo gen 4.6} implies that $c_1\leq \la_i$ for $s_\cF+1\leq i\leq d$; in particular, $c_1\leq \la_{s_\cF+1}$.

\pausa
The previous facts together with Remark \ref{la pinta del nuel} show that $\la(S_\cF)\ua=(c_1\,\uno_{s_\cF},\la_{s_\cF+1},\ldots,\la_d)=\nu(\la\coma\ca)$, according to Definition \ref{feasibilidad}. Moreover, we also get that $\la(S_{\cG_0})=\nu(\la\coma\ca) - \la=\mu(\la\coma\ca)$. Again, since $\cG_0\in\tcal $ we conclude that the majorization relation $\ca\prec\mu(\la\coma\ca)$ holds, and therefore the pair $(\la\coma\ca)$ is feasible. 
As before, Theorem \ref{pre teo caso feasible} shows that $\cG_0$ is a global minimizer of $\Psi_\varphi$ in 
$\tcal $.
\end{proof}
\pausa
The next result is \cite[Proposition 4.5]{mrs3}. Although the result is stated for a global minimum in \cite{mrs3}, the inspection of its proof 
(for $\vfi \in \LLM$, so that the previous results hold) reveals that it also holds for a local minimum as well. Recall from Notations \ref{notas 2} that, if $p>1$ then 
$$K_j = \{ i \in \I_{s_\cF} :  % \la\ua_i(\cF)=
\la_i +\mu_i = c_j\}  
\py J_j = \{i\in \I_k: S_\cF\,g_i = c_j \, g_i\}  
$$ 
for each $j\in \I_p\,$, where $\la = \la\ua(S_{\cF_0})$ and $\mu = \mu\da=\la(S_{\cG_0})$.

\begin{pro}\label{los J ordenados} \rm 
Consider the Notations \ref{notas 2}
with $\cF= (\cF_0\coma \cG_0)\in \cC_\ca^{\rm op}(\cF_0)$ 
and assume that $k\geq d$ and $p>1$. 
Given $i\coma r \in \I_p\,$, $h \in J_i$ and $l \in J_r\,$ then 
$$
i<r \implies a_h-a_l \ge c_i - c_r > 0 
\implies h<l \ .
$$
In particular, there exist $s_0=0<s_1<\ldots<s_{p-1}<s_\cF\leq d$ such that 
\beq 
J_j=\{s_{j-1}+1\coma \ldots\coma s_{j}\} \ , \quad j\in\I_{p-1} 
\py J_p=\{s_{p-1}+1\coma \ldots\coma k\}\ . \QEDP \eeq
\end{pro}

\begin{pro}\label{los la} Consider the Notations \ref{notas 2}
with $\cF= (\cF_0\coma \cG_0)\in \cC_\ca^{\rm op}(\cF_0)$ and assume that $k\geq d$ and $p>1$. We have that:
$$
\barr{rl}
i \in K_1 &\implies 
i< j  \ ( \implies 
\la_i \le \la_j \,)  \peso{for every} 
j \in \bigcup\limits_{r>1} K_r = \I_{s_\cF} \setminus K_1 \ . \earr
$$ 
Inductively, by means of Remark \ref{induc}, 
we deduce that all sets $K_j$ consist of consecutive indexes. 
Therefore, if $s_0=0<s_1<\ldots<s_{p-1}<s_p\stackrel{\rm def}{=}s_{\cF}$ are as in Corollary \ref{los J ordenados} then 
$$ 
K_j=\{s_{j-1}+1\coma \ldots \coma s_{j}\} \ , \quad j\in\I_{p-1} \py K_p=\{s_{p-1}+1,\ldots,s_p\}\ .
$$
\end{pro}

\begin{proof} Assume that there exist $i\in K_1$ and $j\in K_r$ for $1< r$ such that $j<i$. 
In this case, 
$$
\mu_i\leq \mu_j   \ \ \coma \  \ \la_j\leq \la_i  \py  c_1=\la_i+\mu_i> c_r=\la_j+\mu_j\ . 
$$
Consider $\cB=\{v_l\}_{l\in\I_d}$ as in 
Notations \ref{notas 2}. For $t\in[0,1)$ we let 
\beq\label{defi glt}
 g_l(t)=  g_l + \big(\, (1-t^2)^{1/2}-1\,\big)\ \langle g_l,v_i\rangle \ v_i + t \ \langle g_l,v_i\rangle \ v_j \peso{for} l\in \I_k\,. 
\eeq
Notice that, if $l \in J_1\,$, then  
$S_\cF \, g_l = c_1\, g_l \implies \api g_l\coma v_j\cpi =0$. 
Similarly, if $l\in \I_k\setminus J_1$ then $\langle g_l\coma v_i\rangle =0$ (so that  $ g_l(t)=g_l$). 
Therefore the sequence   $\cG(t)=\{g_l(t)\}_{l\in\I_k}\in\tcal $ for $t\in[0,1)$. 
Let $P_i=v_i\otimes v_i$ and $P_{ji}=v_j\otimes v_i$ (so that $P_{ji}\,x=\langle x\coma v_i\rangle \ v_j$). Then, 
for every $t \in [0\coma 1)$,  
$$
g_l(t)= \big(\,I+ ((1-t^2)^{1/2}-1)\ P_i + t\ P_{ji}\,\big)\ g_l\peso{for every} l\in \I_k\ . 
$$
That is, if $V(t)=I+ ((1-t^2)^{1/2}-1)\ P_i + t\ P_{ji}\in\mat$ then $g_l(t)=V(t)\ g_l\,$ for every $l\in \I_k$ and $t\in[0,1)$.
Therefore, we get that
$$
\cG(t)=V(t)\,G=\{V(t)\, g_l\}_{l\in\I_n} \implies S_{G(t)}=V(t)\, S_\cG\, V(t)^*\peso{for} t\in[0,1)\ .
$$
Hence, we obtain the representation 
$$ 
S_{\cG(t)}=\sum_{\ell\in\I_d\setminus\{i,\,j\}}\mu_\ell \ v_\ell\otimes v_\ell 
+ \gamma_{11}(t)\ v_j\otimes v_j+ \gamma_{12}(t)\ v_j\otimes v_i 
+ \gamma_{21}(t)\ v_i\otimes v_j + \gamma_{22}(t)\ v_i\otimes v_i\ ,
$$ 
where the functions $\gamma_{rs}(t)$ are the entries of 
$A(t)=\big(\,\gamma_{rs}(t)\,\big)_{r\coma s=1}^2\in\cH(2)$ 
defined by 
$$
A(t)= \begin{pmatrix} 1 & t \\ 0 & (1-t^2)^{1/2}\end{pmatrix} \begin{pmatrix} \mu_j& 0 \\ 0&\mu_i\end{pmatrix} 
\begin{pmatrix}1& 0\\ t& (1-t^2)^{1/2}\end{pmatrix}  \peso{for every} t\in [0\coma 1)\ .
$$
It is straightforward to check that $\tr(A(t))=\mu_i+\mu_j$ and that $\det(A(t))=(1-t^2)\, \mu_j\,\mu_i\,$. 
These facts imply that 
if we consider the continuous function $L(t)=\la_{\max}(A(t))$ then $L(0)=\mu_j$ and $L(t)$ is strictly 
increasing in $[0,1)$. 
More straightforward computations show that we can consider continuous curves $x_i(t):[0,1)\rightarrow \C^2$ 
which satisfy that  $\{x_1(t),\,x_2(t)\}$ is ONB of $\C^2$ such that
$$A(t)\, x_1(t)=L(t)\, x_1(t) \peso{for} t\in[ 0,1) \py 
x_1(0)=e_1 \ ,\ \ x_2(0)=e_2\,.$$
For $t\in[0,1)$ we let $X(t)=(u_{r,s}(t))_{r,s=1}^2\in \cU(2)$ with columns $x_1(t) $ and $x_2(t)$. 
By construction, $X(t)=[0,1)\rightarrow \cU(2)$ %\mathcal M_2(\C)$ 
is a continuous curve such that $X(0)=I_2\,$ and such that 
$$ X(t)^*\, A(t)\, X(t)=\begin{pmatrix} L(t) & 0 \\ 0 & \mu_i+\mu_j-L(t)\end{pmatrix}\,.$$
Finally, consider the continuous curve $U(t):[0,1)\rightarrow \matud$ given by 
$$
U(t)=  u_{11}(t) \,v_j\otimes v_j+u_{12}(t)\, v_j\otimes v_i + u_{21}(t)\, v_i\otimes v_j + u_{22}(t)\, v_i\otimes v_i
+\sum_{\ell\in\I_d\setminus\{i,\,j\}} v_l\otimes v_l \ .
$$ 
Notice that $U(0)=I$;  also, let $\tilde \cG(t)=U(t)^*\, \cG(t)\in\tcal $ for $t\in[0,1)$, which is a 
continuous curve such that $\tilde \cG(0)=\cG_0\,$.
 In this case, for $t\in[0,1)$ we have that 
$$
S_{\tilde \cG(t)}=U(t)^*\, S_{\cG(t)}\, U(t)
= L(t)\, v_j\otimes v_j + (\mu_i+\mu_j-L(t)) \,v_i\otimes v_i
+\sum_{\ell\in\I_d\setminus\{i,\,j\}}\mu_\ell \ v_\ell\otimes v_\ell \ .
$$ 
In other words, $U(t)$ is constructed in such a way that $\cB=\{v_l\}_{i\in\I_d}$ consists of eigenvectors of $S_{\tilde \cG(t)}$ for 
every $t\in[0,1)$. Hence, if $\tilde \cF(t)=(\cF_0 \coma \tilde \cG(t))$ and  $E(t) = L(t)-\mu_j\geq 0$ for $t\in [0\coma 1)$, we get that 
$$ 
S_{\tilde \cF(t)}= (c_r+E(t)\,)\, v_j\otimes v_j + (c_1- E(t)\,)\, v_i\otimes v_i
+\sum_{\ell\in\I_d\setminus\{i,\,j\}}(\la_\ell + \mu_\ell) \ v_\ell\otimes v_\ell \ .
$$ 
Let $\varepsilon>0$ be such that 
$E(t)=L(t)-\mu_j \leq \frac{c_1-c_r}{2}$ for $t\in[0,\varepsilon]$. (recall that 
$L(0)=\mu_j $ and that $c_1>c_r$). %
Since $L(t)$ (and hence $E(t)$) is strictly 
increasing in $[0,1)$, we see that 
$$
(c_1- E(t)\coma c_r+ E(t))\prec (c_1\coma c_r)\implies \la(S_{\tilde \cF(t)})\prec \la(S_{\cF}) 
\peso{for} t\in(0\coma \varepsilon] \ , 
$$ 
where the majorization relations above are strict. Hence, since $\varphi\in\convfs$ then 
$$ 
\Psi_\varphi(\tilde \cG(t))=\tr(\varphi(\la(S_{\tilde \cF(t)})))< \tr(\varphi(\la(S_{\cF})))
=\Psi_\varphi(\cG_0)\peso{for} t\in(0\coma \varepsilon]\ .
$$
This last fact contradicts the local minimality of $\cG_0$ and the result follows. 
The description of the sets $K_i$'s now follows from Corollary \ref{coro son li}. 
\end{proof}

\section{Local minima are global minima} \label{sec loc son glob}

Throughout this section we adopt Notations \ref{not gen min loc} and \ref{notas 2}. Recall  that we have fixed a 
map $\varphi\in\LLM$ and a sequence 
$\cG_0=\{g_i\}_{i\in\I_k}\in \tcal $ which is 
a {%\bf 
local minimum} of the potential $\Psi_\varphi$ in $\tcal $, 
among several other specific notations.

\pausa
In what follows, we show that local minimizers (as $\cG_0$) of $\Psi_\varphi$ in $\tcal $ are global minimizers (see Theorem \ref{teo locs son globs} below).
In order to do this, we develop a detailed study of the inner structure of local minimizers, based on the results from Section \ref{sec 3 tutti}.

\begin{rem}[Case $k\geq d$] \label{rem caso k mayor d}Consider the Notations \ref{not gen min loc} and \ref{notas 2} and assume that $k\geq d$.
 Then, according to Propositions \ref{los J ordenados} and \ref{los la}, there exist $p\in\I_d$ and $s_0=0<s_1<\ldots<s_{p-1}<s_p=s_\cF\leq d$, where $s_\cF=\rk(S_{\cG_0})$, such that 
\beq\label{los Kj}
\barr{rl}
K_j &=J_j =\{s_{j-1}+1\coma \ldots \coma  s_j\} \ ,\quad 
\peso{for} j\in\IN{p-1}\  , \\&\\
K_p& =\{s_{p-1}+1 \coma \ldots \coma  s_p\} \ , \ J_p=\{s_{p-1}+1\coma \ldots\coma  k\} \ .
\earr 
\eeq
In terms of these indexes we also get that: 
\beq\label{eq el la de F}
\la(S_\cF)=\big(\, c_1 \, \uno_{s_1} \coma 
\dots \coma c_p\, \uno_{s_p-s_{p-1}} \coma \la_{s_p+1} \coma 
\dots \coma \la_d\,\big)\da\in(\R_{>0}^d)\da\peso{if} s_p<d\,
\eeq or 
\beq\label{eq el la de F2}
\la(S_\cF)=\big(\, c_1 \, \uno_{s_1} \coma 
\dots \coma c_p\, \uno_{s_p-s_{p-1}} \,\big)\da\in(\R_{>0}^d)\da\peso{if} s_p=d\,
\eeq 
In what follows, we describe an algorithm that computes both the constants $c_1>\ldots>c_p$ as well as the indexes $s_1<\ldots<s_p$ in terms of the index $s_{p-1}$. \EOE
\end{rem}
\pausa
In order to show the role of the index $s_{p-1}$ as described in Remark \ref{rem caso k mayor d} above, we consider the following
\begin{fed} \label{s feas}\rm 
Let $\la=(\la_i)_{i\in\I_d}\in(\R_{\geq 0}^d)\ua$ and $\ca=(a_i)_{i\in\I_k}\in (\R_{>0}^k)\da$, with $k\geq d$.
\ben
\item Given $0\leq s\leq d-1$ denote by 
$$\la ^s = (\la_{s+1} \coma \dots \coma 
\la_d )\in \R^{d-s}  \py \ca^s = (a_{s+1}\coma \dots \coma a_k) 
\in \R^{k-s}\ ,
$$
the truncations of the vectors $\la$ and $\ca$. 
\item We say that the index $s$ is feasible (for the pair $(\la\coma\ca)$) if 
$(\la^s\coma \ca^s)$ is a feasible pair (see Definition \ref{feasibilidad}) i.e. if $\ca^s\prec \nu(\la^s\coma \ca^s)-\la^s$. 
\EOE 
\een
\end{fed}
\pausa
Notice that, with the notations and terminology from Definition \ref{s feas} above, the pair $(\la\coma\ca)$ is feasible (according to Definition 
\ref{feasibilidad}) if and only if the index $s=0$ is feasible (according to Definition \ref{s feas}).

\pausa
In the following statements we shall use the Notations \ref{not gen min loc} and \ref{notas 2}. Recall  that we have fixed a 
map $\varphi\in\LLM$ and a sequence 
$\cG_0=\{g_i\}_{i\in\I_k}\in \tcal $ which is 
a { local minimum} of the potential $\Psi_\varphi$ in $\tcal $, 
among several other specific notations.
\begin{pro}\label{pro sp-1 es feas} 
Consider the Notations \ref{not gen min loc} and \ref{notas 2} and assume that $k\geq d$.
Let $s_0=0<s_1<\ldots<s_{p-1}<s_p\leq d$ be as in Remark \ref{rem caso k mayor d}. Then:
\begin{enumerate}
\item The index $s_{p-1}\geq 0$ is feasible;
\item The constant $c_p$ and the index $s_p$ are determined by: 
$$c_p=(d-s_{p-1})^{-1}\,(\tr\la^{s_{p-1}}+ \tr\ca^{s_{p-1}}) \py s_p=d \peso{if} (d-s_{p-1})^{-1}\,(\tr\la^{s_{p-1}}+ \tr\ca^{s_{p-1}})\geq \la_d $$ or otherwise, if $(d-s_{p-1})^{-1}\,(\tr\la^{s_{p-1}}+ \tr\ca^{s_{p-1}})<\la_d$ by the identity 
\beq\label{el sp}
\nu(\la^{s_{p-1}}\coma\ca^{s_{p-1}})= \big(\,c_p\, \uno_{s_p-s_{p-1}} \coma \la_{s_p+1} \coma 
\dots \coma \la_d\,\big) \py s_p<d \ .
\eeq
\item If we let $\cG_0^{(p-1)}=\{g_i\}_{i=s_{p-1}+1}^k$ then $\la(S_{\cG_0^{(p-1)}})=((\mu_i)_{i=s_{p-1}+1}^{s_p},0_{d-(s_p-s_{p-1}}))\in(\R^d_{\geq 0})\da$; hence $$(a_i)_{i=s_{p-1}+1}^{k}\prec (\mu_i)_{i=s_{p-1}+1}^{s_p}\,.$$
\end{enumerate}
\end{pro}
\begin{proof}
By Remark \ref{induc} (item $a$) with $j=p-1$) the family $\cG_0^{(p-1)}=\{g_i\}_{i=s_{p-1}+1}^k$ is a local minimum of 
the map
(set $k'=k-s_{p-1}\geq 1$)
$$ \{\cK=\{k_i\}_{i\in \I_{k'}}\in (\cH_{p-1})^{k'}\ , \ \|k_i\|^2=a_{s_{p-1}+i} \, , i\in  \I_{k'}\}\ni \cK \mapsto \text{P}_\varphi(\cF_0^{(p-1)},\cK)$$
where, using the notations from Remark \ref{induc}, $\cH_{p-1} = \big[\, \bigoplus_{i\le p-1} W_i \big]\orto$ and 
$\cF_0^{(p-1)}$ is a sequence in $\cH_{p-1}$ such that $S_{\cF_0^{p-1}}=S_{\cF_0}|_{\cH_{p-1}}$. 
Moreover, by construction of the subspace $\cH_{p-1}$ we see that if we let $\cF^{(p-1)}=(\cF_0^{(p-1)} , \cG_0^{(p-1)})\in \cC_{\ca^{s_{p-1}}}(\cF_0^{(p-1)})$
then $W_p=R(S_{\cG_0^{(p-1)}})$ and $$ S_{\cF^{(p-1)}}\ P_{W_p}=c_p\ P_{W_p}\,.$$
Therefore, by Corollary \ref{lema caso particular}, we see that the pair $(\la^s\coma\ca^s)$ is feasible. The other claims follow from Remark \ref{la pinta del nuel}, Corollary \ref{lema caso particular} and Proposition \ref{frame mayo}.  
\end{proof}

\begin{rem}\label{form sp} 
Observe that, under the assumptions of Proposition \ref {pro sp-1 es feas} then item 2 implies that  
\beq
s_p=\max\{ j\in\I_d:\ \la_j<c_p\}\in\I_d\,. \EOEP
\eeq
\end{rem}

\begin{notas} \rm
Consider the Notations \ref{not gen min loc} and \ref{notas 2}, and assume that $k\geq d$.
\ben
\item We let $h_i:=\la_{i}+a_i$ for every $i\in\I_d$. 
\item Given  $j\leq r \leq d$, let 
\[
P_{j\coma r} =\frac{1}{r-j+1}\ \sum_{i=j}^r\  h_i = 
\frac{1}{r-j+1}\ \sum_{i=j}^r\ \la_{i}+a_i 
\,.
\] 
 We abbreviate $P_{1 \coma r} = P_r\,$ for the initial averages. \EOE  
\een\end{notas}

\pausa
The following result will allow us to obtain several relations between the indexes and constants describing $\la(S_\cF)$ as in Remark \ref{rem caso k mayor d}. We point out that the ideas behind its proof are derived from \cite{mrs3}.

\begin{lem}\label{lema agregado l1}
Consider the Notations \ref{not gen min loc} and \ref{notas 2} and assume that $k\geq d$, $p>1$. With the notations in Remark \ref{rem caso k mayor d} we have that 
\begin{enumerate} 
\item If $1\leq r\leq d$ then 
$$ (a_j)_{j\in\I_r}\prec(P_r-\la_j)_{j\in\I_r} \iff P_r\geq P_i\ , \ \ i\in\I_r \iff P_r=\max\{P_i:\ i\in\I_r\}\,.$$
\item $c_1=P_{s_1}=\max\{P_j: \ j\leq s_{p-1}\}$. Moreover, if $s_1< t\le s_{p-1}
 \implies  P_t< c_1 \,$.
\end{enumerate}
\end{lem}
\begin{proof}
1. Since $\la=\la\ua$ and $\ca=\ca\da$ then $(P_r-\la_j)_{j\in\I_r} =(P_r-\la_j)_{j\in\I_r} \da$ and $(a_j)_{j\in\I_r}=(a_j)_{j\in\I_r}\da$. On the other hand, 
$\sum_{j\in\I_r}a_j=\sum_{r\in\I_r} P_r-\la_j$ by definition of $P_r$. Therefore, 
$(a_j)_{j\in\I_r}\prec(P_r-\la_j)_{j\in\I_r} $ if and only if for $k\in\I_r$ 
$$ \sum_{j\in\I_k} a_j\leq \sum_{j\in\I_k}(P_r-\la_j) \iff 
P_k= \frac{1}{k}\sum_{j\in\I_k} a_j+\la_j\leq P_r \,.$$

\pausa 
2. By Propositions \ref{los J ordenados} and \ref{los la} we see that the sequence $\{g_j\}_{j\in\I_{s_1}}$ is such that its frame operator has eigenvalues given by $(\mu_1,\ldots,\mu_{s_1},0,\ldots,0)\in(\R_{\geq 0}^d)\da$ and their norms are given by $\|g_j\|^2=a_j$ for $j\in\I_{s_1}$. By Proposition \ref{frame mayo} we get that $(a_j)_{j\in\I_{s_1}}\prec(\mu_j)_{j\in\I_{s_1}}$. On the other hand, Proposition \ref{los la} implies that $\la_j+\mu_j=c_1$ for $j\in\I_{s_1}$. Then
$$s_1\, c_1=\sum_{j\in\I_{s_1}}\la_j+\mu_j=\sum_{j\in\I_{s_1}}\la_j+a_j \implies c_1=\frac{1}{s_1}\sum_{j\in\I_{s_1}} \la_j+a_j=P_{s_1}\,.$$
Hence  $(a_j)_{j\in\I_{s_1}}\prec(c_1-\la_j)_{j\in\I_{s_1}}=(P_{s_1}-\la_j)_{j\in\I_{s_1}} \implies 
P_{s_1}=\max\{P_j:\ j\in\I_{s_1}\}$. Consider now $s_1<t\leq s_{p-1}$ and let $2\leq r\leq p-1$ be such that $s_{r-1}<t\leq s_r$. Then
\begin{eqnarray*}
 P_t&=&\frac{s_1}{t} \left( \frac{1}{s_1}\sum_{j\in\I_{s_1}} h_j\right)+\frac{t-s_1}{t} \left( \frac{1}{t-s_1} \sum_{j=s_1+1}^t h_j\right)
\\ &=&\frac{s_1}{t}\ c_1 +\frac{t-s_1}{t} \left( \frac{1}{t-s_1}\ ( \ \sum_{\ell=2}^{r-1} c_\ell \ (s_\ell-s_{\ell-1})+ \sum_{\ell=s_{r-1}+1} ^t \la_\ell +a_\ell\ )\right)\,,
\end{eqnarray*} that represents $P_t$ as a convex combination, where we have used the identities $$ \sum_{i=s_{\ell-1}+1}^{s_{\ell}} h_i=\sum_{i=s_{\ell-1}+1}^{s_{\ell}} \la_i+\mu_i= (s_{\ell}-s_{\ell-1})\ c_{\ell} $$ that follow from the majorization relation $(a_i)_{i=s_{\ell-1}+1}^{s_\ell}\prec (\mu_i)_{i=s_{\ell-1}+1}^{s_\ell}$ for $2\leq \ell\leq p-1$, which are a consequence of Propositions \ref{los J ordenados}, \ref{los la} and \ref{frame mayo}; using the relation $(a_i)_{i=s_{r-1}+1}^{s_r}\prec (\mu_i)_{i=s_{r-1}+1}^{s_r}$, together with the fact that the entries of these two vectors are downwards ordered, we conclude that 
$$\frac{1}{t-s_1}\ ( \ \sum_{\ell=2}^{r-1} c_\ell \ (s_\ell-s_{\ell-1})+ \sum_{\ell=s_{r-1}+1} ^t \la_\ell +a_\ell\ )\leq \frac{1}{t-s_1}\ ( \ \sum_{\ell=2}^{r-1} c_\ell \ (s_\ell-s_{\ell-1})+ c_r\ (t-s_{r-1})  \,)<c_1$$ since the expression to the left is a convex combination of $c_2,\ldots,c_r<c_1$. Finally, we can deduce that  $P_t<\frac{s_1}{t}\ c_1+ \frac{t-s_1}{t} \ c_1=c_1\,.$
\end{proof}

\begin{pro}\label{prop los cs y ss de min loc} \rm
Consider the Notations \ref{not gen min loc} and \ref{notas 2}, and assume that $k\geq d$. Let $p$, $s_0=0<s_1<\ldots<s_{p-1}<s_p\leq d$ and $c_1>\ldots>c_p$ be as in Remark \ref{rem caso k mayor d}, and assume that $p>1$. If we let $\cF=(\cF_0,\,\cG_0)$ then, we have the following relations between these indexes and constants:
\ben
\item The index $s_1 = \max \, \big\{j \le s_{p-1} \, :\, 
P_{1\coma j} = \max\limits_{i\le s_{p-1}}  \, P_{1\coma i} \, \big\}$, and 
$c_1 = P_{1\coma s_1}\,$.
\item If $s_j<s_{p-1}\,$, then
$$
s_{j+1} = \max \, \big\{s_j< r \le s_{p-1} \, :\, 
P_{s_j+1\coma j} = \max\limits_{s_j< i\le s_{p-1}}  \, P_{s_j+1 \coma i} \, \big\} 
\py c_{j+1} = P_{s_j+1\coma s_{j+1}}\ .
$$
\item $s_{p-1}$ is a feasible index and $c_p$ and $s_p$ are determined by (Definition \ref{s feas})
$$c_p=(d-s_{p-1})^{-1}\,(\tr\la^{s_{p-1}}+ \tr\ca^{s_{p-1}}) \py s_p=d \peso{if} (d-s_{p-1})^{-1}\,(\tr\la^{s_{p-1}}+ \tr\ca^{s_{p-1}})\geq \la_d $$ or otherwise, if $(d-s_{p-1})^{-1}\,(\tr\la^{s_{p-1}}+ \tr\ca^{s_{p-1}})<\la_d$ by the identity 
\beq\label{el sp2}
\nu(\la^{s_{p-1}},\ca^{s_{p-1}})= \big(\,c_p\, \uno_{s_p-s_{p-1}} \coma \la_{s_p+1} \coma 
\dots \coma \la_d\,\big) \py s_p<d \ .
\eeq
 Moreover, the following inequalities hold:
\beq\label{eq desi 25}
c_p\geq \frac{1}{\ell-s_{p-1}}\  \sum_{i=s_{p-1}+1}^\ell h_i 
\peso{for} s_{p-1}+1\leq \ell\leq s_p\ . 
\eeq
\een
\end{pro}
\begin{proof}
Item 1 is contained in Lemma \ref{lema agregado l1}. Item 2 above also follows from Lemma \ref{lema agregado l1} applied to the reduced families
$\cG_0^{(j)}$ as defined in Remark \ref{induc}. Notice that, as a consequence of Propositions \ref{los J ordenados} and \ref{los la} then - using the notations from Remark \ref{induc} - we have that, for $1\leq j\leq s_{p-1}$, then $$ I_j=\{i:\ s_j+1\leq i\leq d\} \py  L_j=\{i:\ s_j+1\leq i\leq k\} \,,$$
which imply that $\la^{I_j}=(\la_i)_{i=s_j+1}^d\in(\R_{\geq 0}^{d-s_j})\ua$, $\cG_0^{(j)}=\{g_i\}_{i=s_j+1}^k$ and $\ca^{L_j}=(a_i)_{i=s_j+1}^k\in(\R_{\geq 0}^{k-s_j})\da$.

\pausa
Proposition \ref{pro sp-1 es feas} shows that $s_{p-1}$ is a feasible index and that the constant $c_p$ and the index $s_p$ are determined as described above.
Finally, notice that Proposition \ref{pro sp-1 es feas} shows the majorization relation $(a_i)_{i=s_{p-1}+1}^k\prec (\mu_i)_{i=s_{p-1}+1}^{s_p}$, where $1\leq s_p\leq d\leq k$. Hence, %if $s_{p-1}+1\leq \ell\leq s_p$ then 
$$ 
\sum_{i=s_{p-1}+1}^\ell a_i\leq \sum_{i=s_{p-1}+1}^\ell \mu_i \peso{for every $\ell$ such that} s_{p-1}+1\leq \ell\leq s_p \ . 
$$
Using this inequality and the fact that $\la_i+\mu_i=c_p$ for $s_{p-1}+1\leq i\leq s_p$ we get \eqref{eq desi 25}.
\end{proof}

\pausa
The following are the two main results of \cite{mrs3}. We will need the detailed structure of {\bf global} minima described in both results in order to prove Theorem \ref{teo locs son globs} below.

\begin{teo}[\cite{mrs3}]\label{teo princ jfaa1} \rm 
Let $\cF_0=\{f_i\}_{i\in\I_{n_0}}\in(\C^d)^{n_0}$, let $\la=\la(S_{\cF_0})\ua\in(\R_{\geq 0}^d)\ua$ and let $\ca=(a_i)_{i\in\I_k}\in(\R_{>0}^k)\da$ with $k\geq d$. Define $$ s^*=\min\ \{\ 0\leq s\leq d-1: \ s\text{ is a feasible index for the pair } (\la\coma\ca) \ \}$$ and let $q\in\I_d$, $s_0^*=0<s_1^*<\ldots<s_{q-1}^*=s^*<s_q\leq d$ and $c_1^*<\ldots<c_{q-1}^*<c_q^*$ be computed according to the following recursive algorithm:
\ben
\item The index $s_1^* = \max \, \big\{j \le s^* \, :\, 
P_{1\coma j} = \max\limits_{i\le s^*}  \, P_{1\coma i} \, \big\}$, and 
$c_1^* = P_{1\coma s_1^*}\,$.
\item If the index $s_j^*$ is already computed and $s_j^*<s^*\,$, then
$$
s_{j+1}^* = \max \, \big\{s_j^*< r \le s^* \, :\, 
P_{s_j^*+1\coma j} = \max\limits_{s_j^*< i\le s^*}  \, P_{s_j^*+1 \coma i} \, \big\} 
\py c_{j+1}^* = P_{s_j^*+1\coma s_{j+1}^*}\ .
$$
\item Set $s_{q-1}^*=s^*$, and let $c_q^*$ and $s_{q-1}^*<s_q^*\leq d$ be such that (see Definition \ref{s feas})
$$ 
(c_q^*\,\uno_{s_q^*-s_{q-1}^*} \coma \la_{s_q^*+1} \coma \ldots \coma \la_d)
=\nu(\la^{s^*}\coma\ca^{s^*})\in(\R_{>0}^{d-s^*})\ua\ . 
$$  
\een
Then, there exists $\cF^{\rm op}=(\cF_0,\,\cG^{\rm op})\in\cafo$ with $\la(S_{\cF^{\rm op}})=\nu^{\rm op}(\la\coma\ca)\da$ where 
$$
\nu^{\rm op}(\la\coma\ca) :=(c_1^*\,\uno_{s_1^*} \coma c_2^*\,\uno_{s_2^*-s_1^*} \coma 
\ldots \coma c_{q-1}^*\, \uno_{s^*_{q-1}-s^*_{q-2}} \coma  \nu(\la^{s^*}\coma\ca^{s^*}))\in \R_{>0}^d  \ , 
$$ 
and such that for every $\varphi\in\convfs$ 
\beq\label{el posta} \text{P}_\varphi(\cF_0,\,\cG)\geq \text{P}_\varphi(\cF^{\rm op}) \peso{for every} (\cF_0,\,\cG)\in\cafo\,.
\eeq
Moreover, given $(\cF_0,\,\cG)\in\cafo$, equality holds in Eq. \eqref{el posta} 
$\iff \la(S_{(\cF_0,\,\cG)})=\nu^{\rm op}(\la\coma\ca)\da$.
\QED\end{teo}
\pausa
Consider the notations and terminology from Theorem \ref{teo princ jfaa1} above, and let $\cF^{\rm op}=(\cF_0,\,\cG^{\rm op})\in\cafo$ be such that 
$\la(S_{\cF^{\rm op}})=\nu^{\rm op}(\la\coma\ca)\da$. If $\varp\in\convfs$ then it follows that $\cG^{\rm op}$ is a global minimum of $\Psi_\varp$ so, in particular, $\cG^{\rm op}$ is a local minimum. Hence, we can apply Proposition \ref{prop los cs y ss de min loc} to $\cG^{\rm op}$ and deduce
some of the information contained in Theorem \ref{teo princ jfaa1} with one notable exception, namely that $s_{q-1}=s^*$ is the minimal feasible index of the pair $(\la\coma\ca)$.

\begin{teo}[\cite{mrs3}]\label{teo princ jfaa2} \rm
Let $\cF_0=\{f_i\}_{i\in\I_{n_0}}\in(\C^d)^{n_0}$, let $\la=\la(S_{\cF_0})\ua\in(\R_{\geq 0}^d)\ua$ and let $\ca=(a_i)_{i\in\I_k}\in(\R_{>0}^k)\da$ with $k<d$. Define $$ \tilde \la=(\la_i)_{i\in\I_k} \in(\R_{\geq 0}^k)\ua \py \tilde\nu=\nu^{\rm op}(\tilde \la\coma\ca)\in (\R_{>0}^k)\da\,,$$ where the second vector is constructed 
according to Theorem \ref{teo princ jfaa1} above, and set 
$$
\nu^{\rm op}(\la\coma\ca):=(\tilde \nu \coma \la_{k+1} \coma \ldots \coma \la_d)\in \R_{\geq 0}^d\ .
$$ 
Then there exists $\cF^{\rm op}=(\cF_0,\,\cG^{\rm op})\in\cafo$ with $\la(S_{\cF^{\rm op}})=\nu^{\rm op}(\la\coma\ca)\da$ 
and such that for every $\varphi\in\convfs$ 
\beq\label{el posta2} \text{P}_\varphi(\cF_0,\,\cG)\geq \text{P}_\varphi(\cF^{\rm op}) \peso{for every} (\cF_0,\,\cG)\in\cafo\,.
\eeq
Moreover, given $(\cF_0,\,\cG)\in\cafo$, equality holds in Eq. \eqref{el posta2} 
$\iff \la(S_{(\cF_0,\,\cG)})=\nu^{\rm op}(\la\coma\ca)\da$.
\QED\end{teo}

\begin{rem}
Let $\cF_0=\{f_i\}_{i\in\I_{n_0}}\in(\C^d)^{n_0}$, let $\la=\la(S_{\cF_0})\ua\in(\R_{\geq 0}^d)\ua$ and let $\ca=(a_i)_{i\in\I_k}\in(\R_{>0}^k)\da$. Let
$\nu^{\rm op}(\la\coma\ca)$ be constructed according to Theorems \ref{teo princ jfaa1} or \ref{teo princ jfaa2} depending on the case $k\geq d$ or $k<d$.
The fact that $\nu^{\rm op}(\la\coma\ca)$ is the optimal spectrum for every convex potential $\varphi\in\convfs$ is equivalent to the assertion  that 
\beq\label{hay mayo mil} \nu^{\rm op}(\la\coma\ca)\prec \la(S_{(\cF_0,\,\cG)}) \peso{for every}(\cF_0,\,\cG)\in\cafo\ .
\eeq
See \cite{mrs3} or \cite{FMaP} for an independent proof of this fact. \EOE
\end{rem}

\pausa
The following is our main result:

\begin{teo}\label{teo locs son globs}
Consider the Notations \ref{not gen min loc} and \ref{notas 2}.
Then the local minimizer 
$\cG_0\in\tcal $ is also a  global minimizer of $\Psi_\varphi$ in $\tcal $.
\end{teo}
\begin{proof}
We adopt the terminology of Notations \ref{not gen min loc} and \ref{notas 2}.
We first assume that $k\geq d$ and argue by induction on $p\geq 1$ i.e. the number of constants $c_1>\ldots>c_p>0$.

\pausa Indeed, if $p=1$ then Corollary \ref{lema caso particular} shows that $\cG_0$ is a global minimum of $\Psi_\varphi$ in $\tcal $
 and we are done. Hence, assume that $p>1$ and that the inductive hypothesis holds for $p-1$.
By Proposition \ref{pro sp-1 es feas} the index $s_{p-1}$ is feasible and then, 
$$
s_{p-1}\geq s^*=\min\{\ 0\leq s\leq d-1: \ s \ \text{ is a feasible index of the pair } \ (\la\coma\ca) \ \}\ .
$$
Consider now the notations and terminology from Theorem \ref{teo princ jfaa1}, describing the optimal spectra $\nu^{\rm op}(\la\coma\ca)$ (notice that $q\geq 1$). 

\pausa
Assume first that $q=1$. In this case $\nu= \nu^{\rm op}(\la\coma\ca)=(c_1^* \,\uno_{s_1^*}\coma \la_{s_1^*+1} 
\coma \ldots \coma \la_d)\in(\R_{\geq 0}^d)\ua$. In particular the majorization relation 
$\la (S_\cF)\prec\nu$ in Eq. \eqref{hay mayo mil} shows that  
$$
c_1^*=\min\{\nu_j %\nu^{\rm op}_j(\la\coma\ca)
:\ j\in\I_d\} \geq \min\{\la_j(S_\cF): \ j\in\I_d\}=c_p\ .$$
Hence, by Remark \ref{form sp} and the fact that $\la\leqp \nu \in (\R_{\geq 0}^d)\ua\,$, we deduce that 
$$ 
s_1^*= \max\{j\in\I_d:\ \la_j < c_1^*\}\geq \max\{j\in\I_d:\ \la_j< c_p\}=s_p >s_1 \ .
$$
Hence, by Proposition \ref {teo princ jfaa1},  
$ c_1^*\geq P_{1\coma j}$ for 
$1\leq j\leq s_1^* \implies 
c_1^*\geq P_{1\coma s_1} \stackrel{\ref{prop los cs y ss de min loc}}{=}c_1\,$. 
Using these facts it is easy 
to check that 
$\tr (\nu
)>\tr(\la(S_\cF))$, which contradicts the majorization relation in 
Eq. \eqref{hay mayo mil}.

\pausa
We now assume that $q>1$. In this case, we have that 
\beq\label{eq def s1}
s_1=\max\{1\leq j\leq s_{p-1}: \ P_{1,j}=\max_{1\leq i\leq s_{p-1}} P_{1,i} \} \py c_1=P_{1,s_1}\eeq
and  
\beq\label{eq def s1*}
s_1^*=\max\{1\leq j\leq s_{q-1}^*=s^*: \ P_{1,j}=\max_{1\leq i\leq s_{q-1}^*} P_{1,i} \} \py c_1^*=P_{1,s_1^*}
\eeq
Assume that $s_1^*\neq s_1$. Using that $s^*\leq s_{p-1}$ we see that $s^*=s_{q-1}^*<s_1$. Since $\nu^{\rm op}(\la\coma\ca)$ corresponds to the spectra of a global minimum which, in particular is a local minimum, we can apply item 3 in Proposition \ref{prop los cs y ss de min loc} (see Eq. \eqref{eq desi 25}) and get:
\beq\label{ecuac1}
c_q^*\geq \frac{1}{\ell-s_{q-1}^*}\  \sum_{i=s_{q-1}^*+1}^\ell h_i \peso{for} s_{q-1}^*+1\leq \ell\leq s_q^*\,.
\eeq
We consider the following two sub-cases:

\pausa
Sub-case $a$: $s_q^*\geq s_1$. In this case, since $s^*=s_{q-1}^*<s_1$ we get that
\beq\label{comb conv cs}
c_1=\frac{1}{s_1}\sum_{i=1}^{s_1} h_i= \frac{s^*}{s_1} \ \left( \frac{1}{s^*} \sum_{i=1}^{s^*}h_i \right) + 
\frac{(s_1-s^*)}{s_1} \ \left( \frac{1}{(s_1-s^*)} \sum_{i=s^*+1}^{s_1}h_i \right) 
\eeq that represents $c_1$ as a convex combination of averages. The first average satisfies (by construction of $c_1$ and $s^*\leq s_{p-1}$)
$$ \frac{1}{s^*} \sum_{i=1}^{s^*}h_i \leq c_1 \ \ \implies \ \ \frac{1}{(s_1-s^*)} \sum_{i=s^*+1}^{s_1}h_i\geq c_1\,$$
since otherwise,  Eq. \eqref{comb conv cs} can not hold. Using the hypothesis $s_q^*\geq s_1>s^*=s_{q-1}^*$, Eq. \eqref{ecuac1} 
and the previous inequality
$$ 
c_q^*\geq \frac{1}{s_1-s_{q-1}^*}\  \sum_{i=s_{q-1}^*+1}^{s_1} h_i  \geq c_1\geq c_1^*\,, 
$$ where we have used Eqs. \eqref{eq def s1} and \eqref{eq def s1*} and the fact that $s_{q-1}^*\leq s_{p-1}$. 
Hence $q=1$ contradicting our assumption $q>1$.
 Therefore, Sub-case $a$ is not possible.

\pausa
Sub-case $b$: $s_q^*<s_1$. Recall that $s_{p-1}\geq s^*=s^*_{q-1}$ which, by Eqs. \eqref{eq def s1} and \eqref{eq def s1*}, implies that $c_1\geq c_1^*$.
Thus, $c_1^*\,s_q^*\leq c_1 \,s_q^*<c_1\,s_1$ and hence,
\begin{eqnarray*}
\tr(\nu^{\rm op}(\la\coma\ca)\,)&=&\sum_{i\in\I_q} c_i^*\ (s_{i}^*-s_{i-1}^*)+\sum_{i=s_q^*+1}^d\la_i\leq c_1 \, s_q^*+\sum_{i=s_q^*+1}^d\la_i
\\ &<& \sum_{i\in\I_p} c_i\ (s_{i}-s_{i-1})+\sum_{i=s_p+1}^d \la_i=\tr(S_\cF)\,.
\end{eqnarray*}This last fact
contradicts the majorization relation in Eq. \eqref{hay mayo mil}. We conclude that Sub-case $b$ is not possible.

\pausa
Therefore, we should have that $s_1^*=s_1$ and hence $c_1=c_1^*$. We prove that $\cF=(\cF_0,\,\cG_0)$ is a global minimum by showing that $\la(S_\cF)=\nu^{\rm op}(\la\coma\ca)$.
Indeed, by applying the reduction argument described in Remark \ref{induc} we deduce, setting $k'=k-s_1$, that 
$\cG_0^{(1)}=\{g_{i+s_1}\}_{i\in I_{k'}}$ is a local minimum of 
the map 
\beq\label{ecuac el mapeus}
\{\cK=(k_i)_{i\in \I_{k'}}\in (\cH_{1})^{k'}\ , \ \|g_i\|^2=a_{s_{1}+i} \, , i\in  \I_{k'}\}\ni \cK \mapsto \text{P}_\varphi(\cF_0^{(1)},\cK)
\eeq
where $\cH_{1}=W_1^\perp$ for $W_1=\text{span}\{g_i\}_{i\in\I_{s_1}}$ and $\cF_0^{(1)}$ is a sequence in $\cH_1$ such that $S_{\cF_0^{(1)}}=S_{\cF_0}|_{\cH_1}$.
In this case, by Corollary \ref{coro son li} and the fact that $p>1$, $d'=\dim \cH_1=d-s_1\leq k-s_1=k'$. Moreover, by construction of $\cH_1$, if we let $\tilde \cF=(\cF_0^{(1)},\, \cG_0^{(1)})$ then 
$$ \la(S_{\cF_0^{(1)}})=(\la_{s_1+i})_{i\in\I_{d'}}\da\py \la(S_{\tilde \cF})=(c_2\,\uno_{s_2-s_1},\ldots,c_p\, \uno_{s_p-s_{p-1}},\la_{s_p+1},\ldots,\la_d)\da\,.$$
Hence, the induction hypothesis applies to $\cG_0^{(1)}$ and we conclude that $\cG_0^{(1)}$ is a global minimizer of the map in Eq. \eqref{ecuac el mapeus}.
Therefore, with the notations of Definition \ref{s feas} and Theorem \ref{teo princ jfaa1}, 
	$$ \la(S_{\tilde \cF})=\nu^{\rm op}(\la^{s_1}\coma\ca^{s_1})\,.$$ 
Now, an inspection of the construction in Theorem \ref{teo princ jfaa1} reveals that 
\beq\label{eq palo y}
\nu^{\rm op}(\la\coma\ca)=(c_1^*\,\uno_{s_1^*}\coma\nu^{\rm op}(\la^{s_1}\coma\ca^{s_1}))\in \R_{\geq 0}^d\,. 
\eeq Indeed, since the notion of feasible index depends on the tail of the sequences of eigenvalues and norms we see that  $$ s^*=s_1^*+\min \{0\leq s\leq d'-1: \ s \ \text{ is a feasible index for the pair } \ (\la^{s_1}\coma\ca^{s_1})\}\,.$$
Eq. \eqref{eq palo y} now follows using that $s_1=s_1^*$, the identity above and the formulas for the indexes $s_i^*$ both for $\nu^{\rm op}(\la\coma\ca)$ and 
$\nu^{\rm op}(\la^{s_1}\coma\ca^{s_1})$ from Theorem \ref{teo princ jfaa1}.
Now, we see that 
\beq\label{eq y a la bolsa}
\la(S_\cF)=(c_1\,\uno_{s_1}\coma\la(S_{\tilde \cF}))=(c_1^*\,\uno_{s_1^*}\coma \nu^{\rm op}(\la^{s_1}\coma\ca^{s_1}))=\nu^{\rm op}(\la\coma\ca)\,.
\eeq
Eq. \eqref{eq y a la bolsa} together with Theorem \ref{teo princ jfaa1} show that $\cF=(\cF_0\coma\cG_0)$ is a global minimizer in this case.

\pausa
Finally, in case $k<d$ we argue as in the second part of Remark \ref{induc}. Using Notations \ref{notas 2} and the fact that $\rk(S_{\cG_0})\leq k$, we see that
$\mu_i=0$ for $k+1\leq i\leq d$ and therefore 
\beq\label{eq sobre el espectro de sf en el caso reducido soleado}
S_{\cG_0}=\sum_{i\in\I_k} \mu_i\ v_i\otimes v_i \ \implies \la(S_{\cF})=(\la_1+\mu_1,\ldots,\la_k+\mu_k,\la_{k+1},\ldots,\la_d)\da
\eeq
 and $W:=R(S_{\cG_0})\subset \cH=\text{span}\{v_i:\ i\in\I_k\}$. Notice that $\cH$ reduces $S_{\cF_0}$; then, we can consider a sequence $\tilde \cF_0$ in $\cH$ such that $S_{\tilde \cF_0}=S_{\cF_0}|_\cH$. 
In this case $\cG_0$ is a local minimizer of the map 
\beq\label{eq caso d mayor k}
\T_{\cH}(\ca)=\{\cG=\{g_i\}_{i\in\I_k}\in \cH^k:\ \|g_i\|^2=a_i\ , \ i\in\I_k\ \}
%\tcal 
\ni\cG\mapsto \text{P}_\varphi(\tilde \cF_0,\cG)\,.
 \eeq Since $\dim\cH=k$ then, by the first part of this proof, we conclude that 
$\cG_0$ is a global minimizer of the map in Eq. \eqref{eq caso d mayor k} and that, by Theorem \ref{teo princ jfaa1}, 
\beq\label{eq que despues de un rato se poblo de nubes}
 (\la_1+\mu_1,\ldots,\la_k+\mu_k)\da=\la(S_{(\tilde \cF_0\coma\cG_0)})=\nu^{\rm op}(\tilde \la\coma \ca)\,,
\eeq
where $\tilde \la=\la(S_{\tilde \cF_0})\ua=(\la_i)_{i\in\I_k}$. Finally, notice that by Theorem \ref{teo princ jfaa2} and Eqs. 
\eqref{eq sobre el espectro de sf en el caso reducido soleado}, \eqref{eq que despues de un rato se poblo de nubes} we now conclude that 
\beq\label{eq comieron perdices}
\nu^{\rm op}(\la\coma\ca)=(\nu^{\rm op}(\tilde \la\coma \ca),\la_{k+1},\ldots,\la_d)\da=\la(S_\cF)\,. 
\eeq
Eq. \eqref{eq comieron perdices} together with Theorem \ref{teo princ jfaa2} show that $\cF=(\cF_0\coma\cG_0)$ is a global minimizer in this case.
\end{proof}

% --------------------------------------------------------------------------------------------------------
\section{An application: generalized frame operator distances in $\tcal $}\label{sec aplicados}
% --------------------------------------------------------------------------------------------------------

We begin with a brief description of Strawn's work \cite{strawn} on the frame operator distance. Thus, we consider 
a positive semidefinite $d\times d$ complex matrix $S_0\in\matpos$  and $\ca=(a_i)_{i\in\I_k}\in(\R_{>0}^k)\da$ such that $\ca\prec\la(S_0)$.
The frame operator distance (FOD) is defined as the function $\Theta_2=\Theta_{2\coma S_0\coma \ca}:\tcal \rightarrow \R_{\geq0}$ given by 
$$ \Theta_2(\cG)=\|S_0-S_\cG\|_2 \peso{for} \cG\in \tcal \,,$$ where $\|A\|_2^2=\tr(A^*A)$ for $A\in\mat$ denotes the Frobenius norm in $\mat$.
By Proposition \ref{frame mayo}, the relation $\ca\prec \la(S_{0})$ implies that there exists $\cG^{\rm op}\in \tcal $ such that 
$S_{\cG^{\rm op}}=S_{0}$. In this case, $$ \min\,\{\, \Theta_2(\cG):\ \cG\in \tcal  \,\}=0\,.$$ 
In \cite{strawn}, after noticing the minimum value of $\Theta_2$ above, an algorithm based on approximate gradient descent is presented. This algorithm exploits some geometrical aspects of the differential geometry of the manifold $\tcal $ obtained by Strawn in \cite{strawn2} 
(some of which are of a similar nature to those considered in \cite{mr2010}). Based on numerical evidence, the author then poses the following:

\pausa
{\bf Conjecture} (Strawn \cite{strawn}): Let $S_0\in\matpos$, $\ca=(a_i)_{i\in\I_k}\in(\R_{>0}^k)\da$ with $k\geq d$ and assume that $\ca\prec\la(S_0)$. Then local minimizers of $\Theta_2$ in  $\tcal $ are also global minimizers.
\EOE

\pausa
As remarked in \cite{strawn}, proving this conjecture would provide a beneficial theoretical guarantee for the performance of the frame operator distance algorithm based on approximate gradient descent presented in that paper, from a numerical perspective.

\pausa
In what follows we consider a generalized version of the frame operator distance, in terms of unitarily invariant norms (u.i.n) in $\mat$ (see Section \ref{sec2.2}). Moreover, we consider the general case of $S_0\in\matpos$ and $\ca\in (\R^k_{\geq 0})\da$, without assuming that $k\geq d$ nor $\ca\prec\la(S_0)$.

\begin{fed} \label{defi gen fod}\rm Let $S_0\in\matpos$ and let $\ca=(a_i)_{i\in\I_k}\in(\R_{>0}^k)\da$.
Given a u.i.n. $\nui{\cdot}$ in $\mat$ we consider the generalized frame operator distance (G-FOD) function  
$$\Theta_{(\, S_0\coma\ca\coma\,\nui{\cdot}\  )}=\Theta:\tcal \rightarrow \R_{\geq0} \peso{given by} \Theta(\cG)=\nui{S_{0}-S_{\cG}}\,,$$  where $S_\cG\in\matpos$ denotes the frame operator of $\cG\in \tcal $.
\EOE\end{fed}

\pausa
In case $S_0\in\matpos$ and $\ca=(a_i)_{i\in\I_k}\in(\R_{>0}^k)\da$ are arbitrary then it seems that the minimum value of 
$\Theta$ in $\tcal $ has not been computed in the literature, not even for $\|\cdot\|_2$. 
The following result states that there are structural solution to the G-FOD optimization problem in the sense that there are
families $\cG^{\rm op}\in\tcal $ such that for every u.i.n. $\nui{\cdot}$ the minimum value of $\Theta$ in $\tcal $ is 
$\Theta(\cG^{\rm op})$. In particular, this allows us to compute the minimum value of 
$\Theta$ in $\tcal $ for an arbitrary u.i.n. $\nui{\cdot}$ in the general case.

\pausa
In what follows, given $\tilde \la\in(\R_{\geq 0}^d)\ua$ and $\ca=(a_i)_{i\in\I_k}\in(\R_{>0}^k)\da$ we consider 
$\nu^{\rm op}(\tilde \la\coma\ca)$ constructed as in Theorem \ref{teo princ jfaa1} or Theorem \ref{teo princ jfaa2} according to the case $k\geq d$ or $k<d$.

\begin{teo}\label{teo caract min str-prob}
Let $S_0\in\matpos$ and let $\ca=(a_i)_{i\in\I_k}\in(\R_{>0}^k)\da$. Let $\tilde\la=(\tilde \la_i)_{i\in\I_d}\in(\R_{\geq 0}^d)\ua$ be given by
$\tilde \la=\|S_0\|\,\uno_d-\la(S_0)$ and let $\delta=\delta(\la(S_0)\coma\ca)\in\R^d$ be given by $\delta=(\delta_i)_{i\in\I_d}=\|S_0\| \,\uno_d-\nu^{\rm op}(\tilde \la\coma\ca)$.
\begin{enumerate}
\item For every u.i.n. $\nui{\cdot}$ in $\mat$ we have that 
$$ \min\{ \ \Theta(\cG)=\nui{S_0-S_{\cG}}:\ \cG\in \tcal  \ \}=\nui{D_{\delta}}\,. $$
\item If \ $\nui{\cdot}$ is strictly convex and $\cG\in\tcal $ then
$$\nui{S_0-S_{\cG}}=\nui{D_{\delta}} \peso{if and only if} \la(S_0-S_\cG)=\delta\da\,.$$ In this case, there exists $\{v_i\}_{i\in\I_d}$ an ONB for $\C^d$ such that 
\beq\label{eq teo rep opt str}
S_0=\sum_{i\in\I_d}\la_i(S_0)\ v_i\otimes v_i \py S_{\cG}=\sum_{i\in\I_d} (\la_i(S_0)-\delta_i)\ v_i\otimes v_i\,.
\eeq
\end{enumerate}
\end{teo}
\pausa
We prove Theorem \ref{teo caract min str-prob} below, by means of a translation between the optimization problem for the frame operator distance and the optimization problem for convex potentials of frame completions with prescribed norms. It is worth pointing out that 
a relation between frame operator distances (for the Frobenius norm) and minimization of 
the frame potential of sequences with prescribed norms was already noticed in \cite{strawn}. 

\pausa
We will use the following result for the uniqueness of item 2 in Theorem \ref{teo caract min str-prob} above.
\begin{lem}\label{lem sobre mayo}
Let $a,\, b\in\R^d$ be such that $a\succ b$ and $|a|\da=|b|\da$. Then $a\da=b\da$.
\end{lem}
\begin{proof}
We can assume 
that $a=a\da$ and $b=b\da$.  Also that $a\neq \la \, \uno$ (the case
$a = \la \uno$ is trivial).
We argue by induction on the dimension $d$. If $d=1$ the result is clear.

\pausa
Assume that the result holds for $d-1\geq 1$ and let $a,\, b\in\R^d$ be such that 
 $a\succ b$ and $|a|\da=|b|\da$.
By replacing $a$ by $-a\da $ and $b$ by $-b\da$ 
if necessary, 
we can assume that 
$$
|a_1|\geq |a_d|
\implies \max\limits_{i \in \IN{d}}|a_i| = |a_1| = a_1>0 \ ,
$$
where the fact that $a_1> 0$ (in this case) follows easily  using that $a\da=a \neq \la \, \uno$. 

\pausa
Note that $b\prec a  \implies a_1\geq b_1$. Assume that $a_1>b_1$. Then $a_d = b_d = -a_1\,$. Indeed, $b_d $ must achieve the maximal modulus 
(since $b_1 $ doesn't%in the positive side
), and 
$a_d \le  b_d$ by %the
 majorization. 
Let $\tilde a=(a_i)_{i\in\I_{d-1}}$ and $\tilde b=(b_i)_{i\in\I_{d-1}}$. It is easy to see that 
$\tilde a\da =\tilde a\prec \tilde b=\tilde b\da$ 
and $|\tilde a|\da=|\tilde b|\da$. Hence, by inductive hypothesis $\tilde a=\tilde b\implies a_1=b_1\,$, a contradiction. 

\pausa
So we can assume that 
$a_1=b_1\,$.  As before, we can apply the inductive hypothesis and conclude that $(a_{i+1})_{i\in\I_{d-1}}=(b_{i+1})_{i\in\I_{d-1}}$  and hence $a=b$.
\end{proof}

\begin{proof}[Proof of Theorem \ref{teo caract min str-prob}]
Consider $\tilde S_0=\|S_0\|\,I-S_0\in\matpos$ and let $\cF_0=\{f_i\}_{i\in\I_d}$ be a sequence in $\C^d$ such that $S_{\cF_0}=\tilde S_0$.
Notice that if we let $\la(S_0)=(\la_i(S_0))_{i\in\I_d}\in(\R_{\geq 0} ^d)\da$ then
$\la(S_{\cF_0})\ua=(\|S_0\|-\la_i(S_{0}))_{i\in\I_d}=\tilde \la$.
If $\cG\in\tcal $ then 
\beq\label{eq traduc entre probs}
 S_0-S_\cG= \|S_0\|\,I + (S_0-\|S_0\|\,I)-S_\cG= \|S_0\|\,I-( \tilde S_0+S_\cG )\,.
\eeq
In particular, we get that 
\beq\label{eq ident espec}
\la(S_0-S_\cG)\ua=\|S_0\|\,\uno_d-\la(\tilde S_0+S_\cG)\in(\R^d)\ua\,.
\eeq
But notice that, since $\cG\in\tcal $ then 
  $$\tilde S_0+S_\cG= S_{\cF_0}+S_\cG=S_\cF \peso{for} \cF=(\cF_0,\,\cG)\in\cafo\,. $$
Then, by Theorems \ref{teo princ jfaa1} and \ref{teo princ jfaa2} (according to the case $k\geq d$ or $k<d$)
\beq\label{eq para lema}
\nu^{\rm op}(\tilde \la\coma\ca)\prec \la(\tilde S_0+S_\cG)\ \implies \ \delta\stackrel{\rm def}{=}\|S_0\|\,\uno_d-\nu^{\rm op}(\tilde \la\coma\ca)\prec \|S_0\|\,\uno_d-\la(\tilde S_0+S_\cG)\,.
\eeq
 Using Eq. \eqref{eq ident espec} and that the function $\R\ni x\mapsto |x|\in\R_{\geq0}$ is convex we conclude that 
\beq\label{eq submayo val sing}
 |\delta|=(|\delta_i|)_{i\in\I_d}\prec_w | \, \|S_0\|\,\uno_d-\la(\tilde S_0+S_\cG)\,|=s(S_0-S_\cG)\ua
\eeq
where $s(S_0-S_\cG)=|\la(S_0-S_\cG)|\da\in(\R_{\geq 0}^d)\da$ denotes the vector of singular values of $S_0-S_\cG$.
By Theorem \ref{teo intro prelims mayo}, the previous sub-majorization relation implies that for every uin $\nui{\cdot}$ 
$$ \nui{S_0-S_\cG}\geq \nui{D_\delta} \peso{for every} \cG\in \tcal \,.$$
In order to show that this lower bound is attained, consider $\cG^{\rm op}\in \tcal $ an optimal frame completion in $\tcal $
of $\cF_0$ i.e. such that $\la(\tilde S_0+S_{\cG^{\rm op}})=\la(S_{\cF_0}+S_{\cG^{\rm op}})=\nu^{\rm op}(\tilde \la\coma\ca)\da$. Hence, by Eq. \eqref{eq submayo val sing} we see that $$ s(S_0-S_{\cG^{\rm op}})\ua=|\|S_0\|\,\uno_d- \nu^{\rm op}(\tilde \la\coma\ca)\da|\ \implies \ |\delta|\ua=s(S_0-S_{\cG^{\rm op}})\ua\py 
\nui{S_0-S_{\cG^{\rm op}}}=\nui{D_\delta}\,. $$
This last fact shows that the lower bound is attained at $\cG^{\rm op}$ and proves item 1.

\pausa 
Assume further that $\nui{\cdot}$ is strictly convex and let $\cG\in\tcal $ be such that $\nui{S_0-S_{\cG}}=\nui{D_\delta}\,$.
 The sub-majorization relation in Eq. \eqref{eq submayo val sing} together with the previous hypothesis imply that 
$$ |\delta|\ua=| \, \|S_0\|\,\uno_d-\la(\tilde S_0+S_\cG)\,|=s(S_0-S_\cG)\ua\,.$$ The identity above together with the majorization relation in Eq. \eqref{eq para lema} and Lemma \ref{lem sobre mayo} imply that $$ \delta\da=(\|S_0\|\,\uno_d-\nu^{\rm op}(\tilde \la\coma\ca))\da=(\|S_0\|\,\uno_d-\la(\tilde S_0+S_\cG))\da$$ from which it follows that $\la(S_{\cF_0}+S_\cG)=\la(\tilde S_0+S_\cG)=\nu^{\rm op}(\tilde \la\coma\ca)\da$. Therefore, $\cG\in\tcal $ is a global minimizer of $\Psi_\varphi(\cG)=\tr(\varphi(S_{\cF_0}+S_\cG))$ for every $\varphi\in\convfs$. By Theorem \ref{teo estruc 1}   
there exists $\{v_i\}_{i\in\I_d}$ an ONB of $\C^d$ such that 
\beq\label{eq teo rep opt str2}
\tilde S_0=\sum_{i\in\I_d}\tilde \la_i\ v_i\otimes v_i \py S_{\cG}=\sum_{i\in\I_d} (\nu^{\rm op}(\tilde \la\coma\ca)-\tilde \la)_i\ v_i\otimes v_i\,.
\eeq Using that $S_0=\tilde S_0+\|S_0\|\,I$ so that $\tilde \la= \|S_0\|\,\uno_d-\la(S_0)$ we see that Eq. \eqref{eq teo rep opt str} holds for $\{v_i\}_{i\in\I_d}$. 
\end{proof}

\pausa
The following result settles in the affirmative a generalized version of Strawn's conjecture on local minimizers of the FOD for the Frobenius norm in $\mat$, since we do not assume that $k\geq d$ nor the majorization relation $\ca\prec\la(S_0)$ (see the comments at the beginning of this section).
\begin{teo}
Let $S_0\in\matpos$ and let $\ca=(a_i)_{i\in\I_k}\in(\R_{>0}^k)\da$ and consider the FOD given by 
$$
\Theta_2:\tcal \rightarrow \R_{\geq0} \peso{given by} \Theta_2(\cG)=\|S_{0}-S_{\cG}\|_2\,.
$$ Then, the local minimizers of $\Theta_2$ in $\tcal $ are also global minimizers.
\end{teo}
\begin{proof}
Consider $\tilde S_0=\|S_0\|\,I-S_0\in\matpos$ and let $\cF_0=\{f_i\}_{i\in\I_d}$ be a sequence in $\C^d$ such that $S_{\cF_0}=\tilde S_0$.
Hence $\la(S_{\cF_0})\ua=(\|S_0\|-\la_i(S_{0}))_{i\in\I_d}=\tilde \la$. 
Let $\varphi\in\convfs$ be given by $\varphi(x)=x^2$ for $x\in\R_{\geq 0}$ and  consider 
$\Psi_\varphi:\tcal\rightarrow \R_{\geq 0}$ be given by $\Psi_\varphi(\cG)=\tr(\varphi(S_{(\cF_0\coma \cG)}))=\tr((S_{\cF_0}+S_\cG)^2)$.
If $\cG\in\tcal $ then, using Eq. \eqref{eq traduc entre probs} 
\begin{eqnarray*}
 \Theta_2(\cG)^2&=&\|S_0-S_{\cG}\|_2^2=\tr((\|S_0\|\,I-[S_{\cF_0}+S_\cG])^2)\\ 
&=&\|S_0\|^2\,d - 2\, \|S_0\|\, \tr(S_{\cF_0}+S_\cG) + \tr((S_{\cF_0}+S_\cG)^2)= c+ \Psi_\varphi(\cG)
\end{eqnarray*} where $$c=\|S_0\|\,d+2\,\|S_0\|\,\tr( S_{\cF_0}+S_\cG)=\|S_0\|\,d+2\,\|S_0\|\,(\tr\tilde S_0+\tr \ca)$$ is a constant (since $\cG\in\tcal $).
Hence $\Theta_2(\cG)^2=\Psi_\varphi(\cG)+c$ for every 
$\cG\in\tcal $. In particular, local minimizers of $\Theta_2$ and $\Psi_\varphi$ coincide. The result now follows from these remarks and Theorem \ref{teo locs son globs}.
\end{proof}

\section{Appendix: on a local Lidskii's theorem}\label{appendicity}

\pausa
Let $S\in\matpos$, let $\mu\in(\R_{\geq 0}^d)\da$ and consider $\cO_\mu$  given by 
\beq\label{defi omu App}
\cO_\mu=\{G\in\matpos:\la(G)=\mu\}=\{U^* D_{\mu}\,U:\ U\in\matud\} 
\eeq
We consider the usual metric in $\cO_\mu$ induced by the operator norm; hence $\cO_\mu$ is a metric space.

\pausa
For $\varphi\in\convfs$, let
$\Phi=\Phi_{S,\,\varphi}:\cO_\mu\rightarrow \R_{\geq 0}$ be given by 
\beq\label{defi Phi App}
 \Phi(G)=\tr(\varphi(S+G))=\sum_{j\in\I_d}\varphi(\la_j(S+G))\peso{for} G\in\cO_\mu\,.
\eeq
In what follows, we prove what we call a local Lidskii's theorem (Theorem \ref{teo LLT}) namely 
that local minimizers of $\Phi$ in $\cO_\mu$ are also global minimizers.
\begin{fed}\label{muchas defis mil}
Let $A,\,B\in\matpos$. We consider 
\begin{enumerate}
\item The product manifold $\matud\times \matud$ endowed with the metric $$d((U_1,V_1),(U_2,V_2))=\max\{\|I-U_1^*U_2\|,\,\|I-V_1^*V_2\| \}\,.$$
\item %The function 
$\Gamma=\Gamma_{(A,B)}:\matud\times \matud\rightarrow \matpos_\tau\stackrel{\rm def}{=}\{M\in\matpos:\ \tr(M)=\tau\}$ for $\tau=\tr(A)+\tr(B)$, given by
$$ \Gamma(U,V)=U^* A\,U+V^*B\,V\peso{for} U,\, V\in\matud\,.$$
\item For a given $\varphi\in\convf$ we consider $\Delta_{(A,B)}^\varphi=\Delta:\matud\times \matud\rightarrow \R_{\geq 0}$ given by 
$$\Delta(U,V)=\tr(\varphi(\Gamma(U,V)))\peso{for} U,\, V\in\matud\,.$$
\EOE\end{enumerate}
\end{fed}
\pausa
Our motivation for considering the previous notions comes from the following:
\begin{lem}
\label{lem equiv de los probs}
Let $A=S\in\matpos$ let $\mu\in(\R_{\geq 0}^d)\da$, $B=G_0\in\cO_\mu$ and consider the notations from Definition \ref{muchas defis mil}. Given $\varphi\in\convf$ then
 the following conditions are equivalent:
\begin{enumerate} 
\item $G_0$ is a local minimizer of $\Phi_\varphi$ in $\cO_\mu$;
\item $(I,I)$ is a local minimizer of $\Delta$ on $\matud\times \matud$.
\end{enumerate}
\end{lem}
\begin{proof}
1.$\implies$2. Consider $(U,W)\in \matud\times \matud$ such that 
$$d((U,W),(I,I))=\max\{\,\|I-U^*\|\,,\,\|I-W^*\| \, \}:=\varepsilon\,.$$
Hence, $$ U^* S\,U+W^* G_0\,W=U^*( S+ Z^* G_0\, Z)\,U \peso{with} Z=WU^*\in\matud\,. $$
Notice that 
$$ \|Z-I\|=\|W \,(U^*-W^*)\|\leq \|U^*-I\|+\|I-W^*\|\leq 2\,\varepsilon\,.$$
Hence, $$\Delta(U,W)=\tr(\varphi(U^*( S+ Z^* G_0\, Z)\,U ))= \Phi(Z^*G_0Z)\peso{with} \| Z^*G_0Z-G_0\|\leq 4\,\varepsilon\|G_0\|\,.$$
2.$\implies$1. This is a consequence of the fact that the map $\matud\ni Z\mapsto 
Z^* G\,Z\in\cO_\mu$ is open 
(see, for example,  \cite[Thm. 4.1]{AS}  or \cite{DF}).
\end{proof}

\pausa
In what follows, given $\mathcal S\subset\matpos$ we consider the commutant of $\mathcal S$, denoted $S'$, that is the unital $^*$-subalgebra of $\mat$ given by 
$$ \mathcal S'=\{\ C\in\mat:\ [C,D]=0\ \text{ for every }\ D\in\mathcal S\ \}\subset \mat\,,$$
where $[C,D]=CD-DC$ denotes the commutator of $C$ and $D$.

\pausa
The following result is standard.
\begin{lem}\label{lema gama sobre}
Consider the notations from Definition \ref{muchas defis mil}. Then 
$$
\Gamma
 \peso{is a submersion at} (I,I) \quad \iff \quad %if and only if 
\{A,\,B\}'=\C\cdot I \ .
$$
\end{lem}
\begin{proof}
The (exponential) map $\matsad\ni X\mapsto \exp(X)$ 
allows us to identify the tangent space $\cT_{I}\matud$ with 
$i\cdot\matsad$. Since we consider the product structure on $\matud\times \matud$ we conclude that the differential of $\Gamma$ satisfies
$$D_{(I,I)}\Gamma (X,0)=[A,X] \py D_{(I,I)}\Gamma (0,X)=[B,X] \peso{for} X\in i\cdot \matsad\,. $$ 
Therefore $\Gamma$ is  not a submersion at $(I,I)$ 
if and only if there exists $0\neq Y\in\cT\matpos_\tau$ 
(i.e. $Y\in\matsad$ such that $\tr\, Y=0$)
such that 
\beq\label{orto} 
\tr(Y\,[A,Z])=\tr(Y\,[B,Z])=0 \peso{for every} 
Z\in i\cdot \matsad \ . 
\eeq
Since $\tr(Y\,[A,Z])=\tr([Y,A]\,Z)$ and similarly 
$\tr(Y\,[B,Z])=\tr([Y,B]\,Z)$,  we see that in this case 
$$
[Y,A]=0 =[Y,B]\in i\cdot \matsad \ . 
$$
Moreover, since $Y\neq 0$ and 
$\tr \, Y = 0$, then $Y$ has some non-trivial spectral projection $P$  which also satisfies that $[P,A]=[P,B]=0$. Conversely, in case there exists a non-trivial projection $P$ such that 
$[P,A]=[P,B]=0$, we can construct 
$ Y = \frac{P}{\tr\, P}  - \frac{I-P}{\tr \,(I- P)} $ %\ (I-P)$ 
so that $\tr\, Y =0$. Then 
$0\neq Y\in\cT\matpos_\tau$ and it satisfies Eq. \eqref{orto}, so that this matrix $Y$  is orthogonal to the range of the operator $D_{(I,I)}\Gamma$.
\end{proof}

\begin{pro}\label{pro al final commutan nomas}
Consider the notations from Definition \ref{muchas defis mil} and assume that $\varphi\in\convfs$. If $(I,I)$ is a local minimizer of $\Delta$ in $\matud\times \matud$ then $[A,B]=0$.
\end{pro}
\begin{proof}
Assume that $[A,B]\neq 0$. Then there exists a minimal projection $P$ of the unital $^*$-subalgebra $\mathcal C=\{A,\,B\}'\subseteq\mat$ such that 
$[PA,PB]\neq 0$.  Indeed, $I\in\mathcal C$ is a projection such that $[IA,IB]\neq 0$. If $I$ is not a minimal projection in $\mathcal C$ then
there exists $P_1,\,P_2\in\mathcal C$ non-zero projections such that $I=P_1+P_2$; hence $[P_iA,P_iB]\neq 0$ for $i=1$ or $i=2$. If the corresponding $P_i$ is not minimal in $\mathcal C$ we can repeat the previous argument (halving) applied to $P_i$. Since we deal with finite dimensional algebras, the previous procedure finds a minimal projection $P\in\mathcal C$ as above. By applying a convenient change of orthonormal basis we can assume that 
$R(P)=\text{span}\{e_i: i\in\I_r\}$, where $r=\rk(P)>1$. Since $P$ reduces both $A$ and $B$ we can consider $A_1=A|_{R(P)}\in\mathcal M_r(\C)^+$ and 
$B_1=B|_{R(P)}\in\mathcal M_r(\C)^+$. By minimality of $P$ we conclude that $\{A_1,\,B_1\}'=\mathcal M_r(\C)$.
Using the case of equality of Lidskii's inequality (see Theorem \ref{mrs284}), we conclude that 
$$ b:=(\la(A_1)\da+\la(B_1)\ua)\da\prec a:=\la(A_1+B_1) \py a\neq b\,.$$
If we let $\sigma=\tr(A_1+B_1)$ then, by Lemma \ref{lema gama sobre} the map 
$$\cU(r)\times \cU(r)\ni(U,V)\mapsto U^*A_1\,U+V^*B_1\,V\in\cM_{r}(\C)^+_\sigma $$
is a submersion at $(I_r,I_r)$. In particular, for every open neighborhood $\mathcal N $ of $(I_r,I_r)$ in $\cU(r)\times \cU(r)$ the set 
$$ \cM:=\{ U^*A_1\,U+V^*B_1\,V: \ (U,V)\in\mathcal N \}$$ contains an open neighborhood of $A_1+B_1$ in $\cM_{r}(\C)^+_\sigma$. 
Consider $\rho:[0,1]\rightarrow (\R_{\geq 0}^{r})\da$ given by $\rho(t)=(1-t)\,a+t\,b$ for $t\in[0,1]$. 
Notice that $\rho(t)\prec a$ and $\rho(t)\neq a$ for $t\in (0,1]$. 
If we let $A_1+B_1=W^*D_{a}\,W$ for $W\in\cU(r)$ then the continuous curve $S(\cdot):[0,1]\rightarrow \cM_{r}(\C)^+_\sigma$ given by 
$S(t)=W^*D_{\rho(t)}\,W$ for $t\in[0,1]$ satisfies that $S(0)=A_1+B_1$, $\la(S(t))\prec a$ and $\la(S(t))\neq a$ for $t\in(0,1]$.
Therefore, there exists $t_0\in(0,1]$ such that $S(t)\in\cM$ for $t\in[0,t_0]$ so, in particular, there exists $(U,V)\in\mathcal N$ such that 
$$ S(t_0)=U^*A_1\,U+V^*B_1\,V  \implies
\Delta(U\oplus P^\perp,V\oplus P^\perp)<\Delta(I_d,I_d)  \ , 
$$
because $\varphi\in\convfs$, where $U\oplus P^\perp,\, V\oplus P^\perp\in\matud$ act as the identity on $R(P)^\perp\subset \C^d$.
Since $\cN$ was an arbitrary neighborhood of $(I_r,I_r)$ we conclude that $(I_d,I_d)$ is not a local minimizer of $\Delta$ in $\matud\times \matud$.
\end{proof}

\begin{teo}[Local Lidskii's theorem]\label{teo LLTApp}
Let $S\in\matpos$ and $\mu=(\mu_i)_{i\in\I_d}\in(\R_{\geq 0}^d)\da$. Assume that 
$\varp \in \convfs$ and that $G_0\in\cO_\mu$ is a local minimizer of $\Phi_{S\coma \varp}$ on $\cO_\mu\,$. 
Then, there exists $\{v_i\}_{i\in\I_d}$ an ONB of $\C^d$ such that, 
if we let $(\la_i)_{i\in\I_d}=\la\ua(S)\in (\R_{\geq 0}^d)\ua$ then 
\beq\label{eq base buenapp}
 S=\sum_{i\in\I_d}\la_i\ v_i\otimes v_i \py G_0=\sum_{i\in\I_d}\mu_i\ v_i\otimes v_i\ .
\eeq
 In particular, $\la(S+G_0)=(\la(S)\ua+\la(G_0)\da)\da$ so $G_0$ is also a global minimizer of $\Phi$ on $\cO_\mu\,$.
\end{teo}
\begin{proof}
By Lemma \ref{lem equiv de los probs} and Proposition \ref{pro al final commutan nomas} we conclude that $[S,G_0]=0$. 
Notice that in this case there exists $\cB=\{v_i\}_{i\in\I_d}$ an ONB of $\C^d$ such that 
$$
S=\sum_{i\in\I_d} \la_i \ v_i\otimes v_i
\ , \ G_0=\sum_{i\in\I_d} \nu_i \ v_i\otimes v_i \ \text{ with }\ \la=(\la_i)_{i\in\I_d}\in (\R_{\geq 0}^d)\ua \ ,
$$ 
for some $\nu_1,\ldots,\nu_d\geq 0$.
We now show that under suitable permutations of the elements of $\cB$ we can obtain a representation as in Eq. \eqref{eq base buenapp} above.
Indeed, assume that $j\in\I_{d-1}$ is such that $\nu_j<\nu_{j+1}$. If we assume that $\la_j<\la_{j+1}$ then consider the continuous curve of unitary operators $U(t):[0,\pi/2)\rightarrow \matud$ given by $$U(t)=\sum_{i\in\I_d\setminus\{j,\,j+1\}}v_i\otimes v_i + \cos(t)\ (v_j\otimes v_j +v_{j+1}\otimes v_{j+1}) +\sin(t)\ (v_j\otimes v_{j+1}- v_{j+1}\otimes v_j)\ , \ \ t\in[0,\pi/2)\,.$$
Notice that $U(0)=I_d$. We now define the continuous curve $G(t)=U(t) \,G_0\,U(t)^* \in\cO_\mu$, for $t\in[0,\pi/2)$.
Then $G(0)=G_0$ and we have that  
\beq\label{eq SFt}
 S+G(t)=\sum_{i\in\I_d\setminus\{j,\,j+1\}} (\la_i+\nu_i) \ v_i\otimes v_i + \sum_{r,s=1}^2 \gamma_{r,s}(t) \ v_{j+r}\otimes v_{j+s}\,,\eeq
where $M(t)=(\gamma_{r,s})_{r,s=1}^2$ is determined by  $$ M(t)
= \begin{pmatrix} \la_j & 0 \\0 & \la_{j+1}\end{pmatrix}+ V(t)\begin{pmatrix}\nu_{j} & 0 \\0 & \nu_{j+1}\end{pmatrix}
V(t)^* \py V(t)=\begin{pmatrix}\cos(t) & \sin(t) \\-\sin(t) & \cos(t)\end{pmatrix} \ , \ \ t\in[0,\pi/2)\,.
$$
Let us consider 
\beq\label{defi R}R(t)=
V^*(t) \begin{pmatrix}
\la_j - \la_{j+1} & 0 \\
0 & 0
\end{pmatrix}V(t)
+ 
\begin{pmatrix}
\nu_j & 0 \\
0 & \nu_{j+1}
\end{pmatrix}
\implies M(t)=V(t) \, R(t)\, V^*(t) + \la_{j+1}\, I_2\,.
\eeq
We claim that $\la(R(t))\prec \la(R(0))$ and $\la(R(t))\neq \la (R(0))$ for $t\in(0,\pi/2)$ (i.e., the majorization relation is strict). Indeed, since 
$R(t)$ is a curve in $\mathcal M_2(\C)^+$ such that $\tr(R(t))$ is constant, it is enough to show that the function $[0,\pi/2)\ni t\mapsto \tr(R(t)^2)$ is strictly decreasing in $[0,\pi/2)$. 
Indeed, since $\la_j-\la_{j+1}>0$ then 
$$ V^*(t) \begin{pmatrix}\la_j - \la_{j+1} & 0 \\ 0 & 0 \end{pmatrix}V(t) = g(t)\otimes g(t) \peso{where} g(t)=(\la_j-\la_{j+1})^{1/2}(\cos(t),\sin(t)) \ , \ \ t\in[0,\pi/2)\,.
$$ If $D\in\cM_2(\C)$ is the diagonal matrix with main diagonal $(\nu_j,\,\nu_{j+1})$ then 
$R(t)=g(t)\otimes g(t)+D$ so 
$$\tr(R(t)^2)=\tr((g(t)\otimes g(t))^2)+\tr(D^2)+2\,\tr(g(t)\otimes g(t)\ D)=c+\langle D\,g(t),\,g(t)\rangle$$
where $c=\|g(t)\|^4+\nu_j^2+\nu_{j+1}^2=(\la_j-\la_{j+1})^2+\nu_j^2+\nu_{j+1}^2\in\R$ is a constant and 
$$\langle D\,g(t),\,g(t)\rangle =(\la_j-\la_{j+1})\ (\cos^2(t) \, \nu_j+\sin^2(t)\,\nu_{j+1})$$ is strictly decreasing in $[0,\pi/2)$, since $\nu_j>\nu_{j+1}$. Thus, $\la(R(t))\prec \la(R(0))$ and $\la(R(t))\neq \la (R(0))$ for $t\in(0,\pi/2)$. Hence, by Eq. \eqref{defi R}, we see that  $$ \la(M(t))=\la(R(t))+\la_{j+1}\,\uno_2 \implies \la(M(t))\prec \la(M(0)) \ , \ \la(M(t))\neq \la(M(0)) \ ,\ \ t\in(0,\pi/2)\,.$$ Using Eq. 
\eqref{eq SFt} and that $\varphi\in\convfs$, we see that for $t\in(0,\pi/2)$
$$ \Phi(G(t))=\sum_{i\in\I_d\setminus\{j,\,j+1\}} \varphi(\la_i+\nu_i)+\tr(\varphi(M(t)))< \sum_{i\in\I_d\setminus\{j,\,j+1\}} \varphi(\la_i+\nu_i)+\tr(\varphi(M(0)))=\Phi(G(0))$$ This last inequality, which is a consequence of the assumption $\la_j<\la_{j+1}$, contradicts the local minimality of $G_0$ in $\cO_\mu$. Hence, since $\la_j\leq \la_{j+1}$ we see that $\la_j=\la_{j+1}$; in this case, we can consider the basis $\cB'=\{v_i'\}_{i\in\I_d}$ obtained by transposing the vectors $v_j$ and $v_{j+1}$ in the basis $\cB$. In this case $S\ v_i'=\la_i\ v_i'$ for $i\in\I_d$, $G_0\ v_i=\nu_i\ v_i'$ for $i\in\I_d\setminus\{j,\,j+1\}$ and $G_0\ v_j'=\nu_{j+1}\, v_j'$, $G_0 v_{j+1}'=\nu_{j}\, v_{j+1}'$. After performing this argument at most $d$ times we get the desired ONB.
\end{proof}

% ------------------------------------------------------------------------------
% ------------------------------------------------------------------------------
%\renewcommand{\refname}{Referencias}

% ------------------------------------------------------------------------------
% ------------------------------------------------------------------------------

\end{document}